\documentclass[article]{elsarticle}

\usepackage{lineno,hyperref}
\modulolinenumbers[5]

\newcommand{\putaway}[1]{}

\usepackage{amsthm}

\usepackage{amsmath,amssymb,tikz,bbm,mathrsfs}
\usetikzlibrary{arrows,decorations.pathmorphing,backgrounds,positioning,fit,petri}

\newtheorem{theorem}{Theorem}[section]
\newtheorem{lemma}[theorem]{Lemma}
\newtheorem{proposition}[theorem]{Proposition}

\newtheorem{corollary}[theorem]{Corollary}
\newtheorem{definition}[theorem]{Definition}
\newtheorem{example}[theorem]{Example}
\newtheorem{strategy}[theorem]{Strategy}
\newtheorem{remark}[theorem]{Remark}

\usepackage{nicefrac,amsmath,amssymb}

\newcommand{\pto}{\to}

\usepackage{amsfonts}
\usepackage{bm}
\usepackage[inline]{enumitem}
\usepackage{stmaryrd,multicol,tikz,float}
\usepackage{MnSymbol}

\newcommand{\pointed}[1]{{\bm #1}}

\newcommand{\bis}{\leftrightarroweq}
\newcommand{\oper}{\rho}
\newcommand{\lft}{\mathfrak L}
\newcommand{\rgt}{\mathfrak R}

\newcommand{\meg}{{\sc meg}}
\newcommand{\valfun}[3]{\lambda #1.#2 [#3]}
\newcommand{\trC}[1]{t_{\forall}^{#1}}
\newcommand{\trtanpd}{t_{\pt}^\bot}
\newcommand{\trtanwedge}[1]{t_{\pt}^{#1}}
\newcommand{\trpsmu}{t_\ps^\mu}

\newcommand{\trmutan}{t_\mu^{\pt}}
\newcommand{\trcl}[1]{t_\ps^\pd({#1})}

\newcommand{\trtan}[1]{t_{\pt}^\mu({#1})}
\newcommand{\trtanb}{t_{\pt}^\mu}
\newcommand{\pt}{\pd^\ast}
\newcommand{\nt}{\nd^\ast}
\newcommand{\frm}[1]{{\mathcal{#1}}} 
\newcommand{\cls}[1]{{\mathbf{#1}}} 
\newcommand{\dom}[1]{{|#1|}} 
\newcommand{\rel}[1]{R_{#1}} 
\newcommand{\V}{V} 
\newcommand{\rgnA}{{A}} 
\newcommand{\rgnB}{{B}} 
\newcommand{\rgnX}{{X}} 
\newcommand{\power}[1]{2^{#1}}
\newcommand{\PV}{P}

\newcommand{\nc}{{\boxplus}}
\newcommand{\ps}{{\diamondplus}}
\newcommand{\nd}{{\Box}}
\newcommand{\pd}{{\Diamond}}

\newcommand{\peq}{\preccurlyeq}

\newcommand{\lang}{{\mathcal L}}

\def\lb{\left\llbracket}
\def\rb{\right\rrbracket}
\newcommand{\val}[1]{\lb #1 \rb}

\newcommand{\david}[1]{\footnote{{\color{orange} DAVID: #1}}}
\newcommand{\petar}[1]{\footnote{{\color{blue} PETAR: #1}}}

\begin{document}

\begin{frontmatter}

\title{Succinctness in subsystems of the spatial $\mu$-calculus}

\author{David F\'ernandez-Duque\fnref{davidfoot}}
\author{Petar Iliev\fnref{petarfoot}}
\address{Institute de Recherche en Informatique de Toulouse, Toulouse University}
\fntext[davidfoot]{Email: david.fernandez@irit.fr}
\fntext[petarfoot]{Email: petar.iliev@irit.fr}




\begin{abstract}
In this paper we systematically explore questions of succinctness in modal logics employed in spatial reasoning.
We show that the closure operator, despite being less expressive, is exponentially more succinct than the limit-point operator, and that the $\mu$-calculus is exponentially more succinct than the equally-expressive tangled limit operator. These results hold for any class of spaces containing at least one crowded metric space or containing all spaces based on ordinals below $\omega^\omega$, with the usual limit operator. We also show that these results continue to hold even if we enrich the less succinct language with the universal modality.
\end{abstract}


\end{frontmatter}


\section{Introduction}

In spatial reasoning, as in any other application of logic, there are several criteria to take into account when 
choosing an appropriate formal system. A more expressive logic has greater potential applicability, but often 
at the cost of being less tractable.
Similarly, a more {\em succinct} logic is preferable, for example, when storage
capacity is limited: even when two formal languages $\lang_1$ and $\lang_2$ are equally expressive, it may be the case 
that certain properties may be represented in $\lang_1$ by much shorter expressions than in $\lang_2$. 
As we will see, this is sometimes the case even when $\lang_1$ is strictly {\em less} expressive than $\lang_2$.

Qualitative spatial reasoning deals with regions in space and abstract relations between them, without requiring a precise description of them.
 It is useful in settings where data about such regions is incomplete or, otherwise, highly complex, yet precise numerical values of 
coordinates are not necessary: in such a context, qualitative descriptions may suffice and can be treated more efficiently 
from a computational perspective.
One largely unexplored aspect of such 
efficiency lies in the succinctness of the formal languages employed.
To this end, our goal is to develop a 
first approximation to the study of succinctness in the context of modal logics of space.

\subsection{State-of-the-art in succinctness research}

Succinctness is an important research topic that has been quite active for the last couple of decades.
For example, it was shown by Grohe and Schweikardt \cite{groheschw05}
that the four-variable fragment of first-order logic is exponentially more succinct than the three-variable one
on linear orders, while Eickmeyer et al.~\cite{eickmeyer} offer a study of the succinctness of order-invariant 
sentences of first-order  and monadic second-order 
logic on graphs of bounded tree-depth. Succinctness problems for temporal 
logics for formal verification of programs were studied, among others, by
Wilke \cite{Wilke}, Etessami et al.~\cite{etassamivardiwilke}, and Adler and Immerman \cite{adlerimmerman}, while it was convincingly argued 
by Gogic et al.~\cite{gogicetal} that, as far as knowledge representations formalisms studied in the artificial
intelligence are concerned, succinctness offers a more fine-grained comparison criterion than
expressivity or computational complexity.

Intuitively, proving that one language $\lang_1$ is more succinct than another language $\lang_2$ ultimately boils down 
to proving a sufficiently big  lower bound on the size of $\lang_2$-formulas expressing some semantic property. 
For example, if we want to show that $\lang_1$ is exponentially\footnote{Analogous considerations apply in the 
case when we want to show that $\lang_1$ is  doubly exponentially or non-elementarily etc.
more succinct than $\lang_2$.}  more succinct
than $\lang_2$, we have to find an infinite sequence of semantic properties (i.e., classes of models) ${\cls P}_1, {\cls P}_2, \ldots $ 
definable in both $\lang_1$ and $\lang_2$, show that there are $\lang_1$-formulas $\varphi_1, \varphi_2, \ldots$ defining
${\cls P}_1, {\cls P}_2, \ldots $ and prove that, for every $n$, every $\lang_2$-formula $\psi_n$ defining ${\cls P}_n$
has size exponential in the size of $\varphi_n$. Many such lower bound proofs, especially in the setting of temporal logics,
rely on automata theoretic arguments possibly combined with complexity theoretic assumptions. In the present paper, we
use formula-size games that were developed in the setting of Boolean function complexity by Razborov \cite{razborovcombinatorica}
and in the setting of first-order logic and some temporal logics by Adler and Immerman \cite{adlerimmerman}. By now, the formula-size games 
have been adapted to a host of modal logics (see for example French et al.~\cite{ijcai}, Hella and Vilander \cite{hellaaiml}, Figueira and Gor{\'i}n \cite{gorinfigueira}, van der Hoek et al.~\cite{aamas}) 
and used to obtain lower bounds on modal formulas expressing properties of Kripke models.
Our goal is to build upon these techniques in order to apply them to modal logics employed in spatial reasoning.


\subsection{Spatial interpretations of modality}\label{SecModSpace}

Modal logic is an extension of propositional logic with a `modality' $\pd$ and its dual, $\nd$, so that if $\varphi$ is any formula,
$\pd \varphi$  and $\nd\varphi$  are formulas too. There are several interpretations for
 these modalities, but one of the first was studied by McKinsey and Tarski \cite{McKinsey1944}, who proposed a topological reading for them. 
These semantics have regained interest in the last decades, due to their potential for spatial reasoning, especially when modal logic is augmented with a universal modality as studied by Shehtman \citep{ShehtmanEverywhereHere}, or fixpoint operators studied by the first author \cite{FernandezIJCAI} and Goldblatt and Hodkinson \cite{Goldblatt2017Spatial}.

The intention is to interpret formulas of the modal language as subsets of a `spatial' structure, such as $\mathbb R^n$. To do this, we use the {\em closure} and the {\em interior} of a set $\rgnA\subseteq \mathbb R^n$. The closure of $\rgnA$, denoted $c(\rgnA)$, is the set of points that have distance zero from $\rgnA$; its interior, denoted $i(\rgnA)$, is the set of points with positive distance from its complement. To define these, for $x,y\in\mathbb R^n$, let $\delta(x,y)$ denote the standard Euclidean distance between $x$ and $y$. It is well-known that $\delta$ satisfies
\begin{enumerate*}[label=(\roman*)]
\item $\delta(x,y)\geq 0$,
\item $\delta(x,y)=0$ iff $x=y$,
\item $\delta(x,y)=\delta(y,x)$ and
\item the triangle inequality, $\delta(x,z)\leq \delta(x,y)+\delta(y,z)$.
\end{enumerate*}
More generally, a set $X$ with a function $\delta\colon X\times X\to \mathbb R$ satisfying these four properties is a {\em metric space.} The Euclidean spaces $\mathbb R^n$ are metric spaces, but there are other important examples, such as the set of continuous functions on $[0,1]$ (with a suitable metric).

\begin{definition}
Given a metric space $X$ and $\rgnA\subseteq X$, we say that a point $x$ has {\em distance zero from $\rgnA$} if for every $\varepsilon>0$, there is $y \in \rgnA$ so that $\delta(x,y)<\varepsilon$. If $x$ does not have distance zero from $a$, we say it has {\em positive distance} from $\rgnA$. Then, $c(\rgnA)$ is the set of points with zero distance from $\rgnA$ and $i(\rgnA)$ is the set of points with positive distance from its complement.
\end{definition}

Note that if we denote the complement of $A$ by $\overline A$, then we have that $i(\rgnA)=\overline{c(\overline {\rgnA})}$.
The basic properties of $c$ are well-known, and we mention them without proof.

\begin{proposition}
If $X$ is a metric space and $c$ is the closure operator on $X$, then, given sets $\rgnA,\rgnB\subseteq X$,
\begin{enumerate}[label=(\roman*)]
\item $c(\varnothing)=\varnothing$,
\item $\rgnA\subseteq c(\rgnA)$,
\item $c(\rgnA)=c(c(\rgnA))$ and
\item $c(\rgnA\cup \rgnB)=c(\rgnA)\cup c(\rgnB)$.
\end{enumerate}
\end{proposition}

We will say that any set $X$ equipped with a function $c\colon 2^X\to 2^X$ satisfying these four properties is a {\em closure space,} and that $c$ is a closure operator. Closure spaces are simply topological spaces in disguise, but presenting them in this fashion will have many advantages for us. To be precise, if $(X,c)$ is a closure space and $A = c(A)$, we say that $X$ is {\em closed,} and its complement is {\em open;} the family of open sets then gives a topology in the usual way.\footnote{We will not define topological spaces in this text, and instead refer the reader to a text such as \cite{Munkres}.}

 From a computational perspective, it can be more convenient to work with closure spaces than with metric spaces, as finite, non-trivial closure spaces can be defined in a straightforward way, and thus spatial relations can be represented using finite structures. To be precise, let $W$ be a set and $R\subseteq W\times W$ be a binary relation; the structure $(W,R)$ is a {\em frame.} Then, if $R$ is a preorder (i.e., a transitive, reflexive relation), the operator $R^{-1}[\cdot]$ defined by $R^{-1}[\rgnA]=\{w\in W:\exists v\in \rgnA \ (w \mathrel R v)\}$ is a closure operator.

A good deal of the geometric properties of regions in a metric space $X$ are reflected in the behavior of its closure operator; however, some information is inevitably lost. It has been observed that more information about the structure of $X$ is captured if we instead consider its {\em limit operator.} For $\rgnA\subseteq X$, define $d(\rgnA)$ to be the set of points such that, for every $\varepsilon>0$, there is $y\in a$ {\em different} from $x$ such that $\delta(x,y)<\varepsilon$. It is no longer the case that $\rgnA\subseteq d(\rgnA)$: for example, if $\rgnA=\{x\}$ consists of a single point, then $d(\rgnA)=\varnothing$.

 Nevertheless, $d$ still satisfies the following properties.
 
\begin{proposition}
Let $X$ be a metric space, and let $d \colon \power X\to \power X$ be its limit operator. Then, for any $\rgnA,\rgnB\subseteq X$, 
\begin{enumerate}[label=(\roman*)]
\item $d(\varnothing)=\varnothing$,
\item $d(d(\rgnA))\subseteq d(\rgnA)$,\footnote{If, instead, we let $X$ be an arbitrary topological space, then only the weaker condition $d(d(\rgnA))\subseteq d(\rgnA)\cup \rgnA$ holds in general.} and
\item $d(\rgnA\cup \rgnB)=d(\rgnA)\cup d(\rgnB)$.
\end{enumerate}
\end{proposition}

In order to treat closure operators and limit operators uniformly, we will define a {\em convergence space} to be a pair $(X,d)$, where $d\colon 2^X\to 2^X $ satisfies these three properties (see Definition \ref{DefModalSpace}). In any convergence space, we can then define $c(\rgnA)=\rgnA\cup d(\rgnA)$, but in general, it is not possible to define $d$ in terms of $c$ using Boolean operations. In particular, the {\em isolated points} of $\rgnA$ can be defined as the elements of $\rgnA\setminus d(\rgnA)$, but they cannot be defined in terms of $c$ (see Figure \ref{FigCID}). As before, a convenient source of convergence spaces is provided by finite frames: if $(W,R)$ is such that $R$ is transitive (but not necessarily reflexive), then $(W,R^{-1}[\cdot])$ is a convergence space.

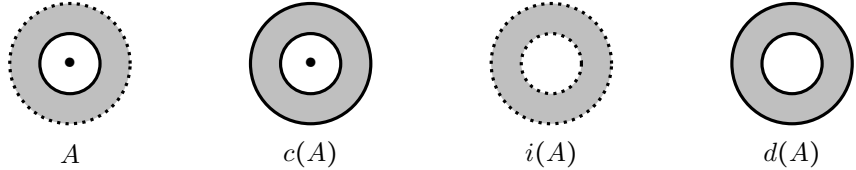
\begin{figure}

\begin{center}

\begin{tikzpicture}[scale=.8]
\draw[fill=gray!50,very thick, dotted] {(0,0) circle (1)};

\draw[fill=white,very thick] {(0,0) circle (.5)};

\draw[very thick] {(0,0) node {$\bullet$}};

\draw[very thick] {(0,-1.5) node {$\rgnA$}};

\draw[fill=gray!50,very thick] {(4,0) circle (1)};

\draw[fill=white,very thick] {(4,0) circle (.5)};

\draw[very thick] {(4,0) node {$\bullet$}};

\draw[very thick] {(4,-1.5) node {$c(\rgnA)$}};

\draw[fill=gray!50,very thick,dotted] {(8,0) circle (1)};

\draw[fill=white,very thick,dotted] {(8,0) circle (.5)};

\draw[very thick] {(8,-1.5) node {$i(\rgnA)$}};

\draw[fill=gray!50,very thick] {(12,0) circle (1)};

\draw[fill=white,very thick] {(12,0) circle (.5)};

\draw[very thick] {(12,-1.5) node {$d(\rgnA)$}};

\end{tikzpicture}

\end{center}

\caption{A region in $\mathbb R^2$ and its closure, interior and limit sets. The point in the middle is the only isolated point of $\rgnA$.}\label{FigCID}

\end{figure}


Logics involving the $d$ operator are more expressive than those with $c$ alone, for example being able to distinguish the real line from higher-dimensional space and Euclidean spaces from arbitrary metric spaces, as shown by Bezhanishvili et al.~\cite{Bezhanishvili2005}. Nevertheless, as we will see, modal logic with the $c$-interpretation of modal logic is exponentially more succinct than modal logic with the $d$-interpretation.

\subsection*{Layout of the article} In Section \ref{SecMu}, we review the syntax and semantics of the spatial $\mu$-calculus, extended in Section \ref{SecExtLan} to include the closure and tangled limit operators. Section \ref{SecTruthPres} then reviews some model-theoretic operations that preserve the truth of modal formulas, and Section \ref{SecSpecial} discusses the classes of spaces that we will be interested in.

Section \ref{SecGames} presents the game-theoretic techniques that we will use to establish our core results in Section \ref{SecSuccCl}, where it is shown that the closure operator is more succinct than the limit-point operator. This is extended to include the tangled limit operator and the universal modality in Section \ref{SecSuccExt}, leading to our main results, namely Theorem \ref{TheoMain} and Theorem \ref{TheoMuSucc}. Finally, Section \ref{SecConc} provides some concluding remarks and open problems.

\section{The spatial $\mu$-calculus}\label{SecMu}

In this section we present the modal $\mu$-calculus and formalize its semantics over {\em convergence spaces,} a general class of models that allow us to study all the semantic structures we are interested in under a unified framework. Let us begin by defining the basic formal language we will work with.

\subsection{Syntax}

We will consider logics over variants of the language $\lang_{\pd \forall }^{\mu}$ given by the following grammar (in Backus-Naur form). Fix a set $\PV$ of propositional variables, and define:
\[\varphi,\psi:= \hfill  \top \hfill  \mid \hfill  \bot \hfill  \mid \hfill  p \hfill  \mid \hfill  \overline p \hfill  \mid \hfill  \varphi \wedge \psi \hfill \mid \hfill \varphi \wedge \psi  \hfill  \mid  \hfill \pd \varphi   \hfill \mid \hfill \nd \varphi  \hfill \mid  \hfill \forall \varphi  \hfill \mid  \hfill \exists \varphi  \hfill \mid  \hfill \mu p. \varphi  \hfill \mid  \hfill \nu p. \varphi\]
Here, $p\in \PV$, $\overline p$ denotes the negation of $p$, and $\overline p$ may not occur in $\mu p. \varphi$ or $\nu p. \varphi$. For the game-theoretic techniques we will use, it is convenient to allow negations only at the atomic level, and thus we include all duals as primitives, but not negation or implication; however, we may use the latter as shorthands, defined via de Morgan's laws. As usual, formulas of the forms $p,\overline p$ are {\em literals.} It will also be crucial for our purposes to measure the size of formulas: the {\em size} of a formula $\varphi$ is denoted $|\varphi|$ and is 
defined as the number of nodes in its syntax tree.

\begin{definition}\label{defComplexity}
We define a function ${|\cdot|} :  \lang^\mu_{\pd\forall} \to \mathbb N$ recursively by
\begin{itemize}
\item $|p|=|\overline p|=1$

\item $|\varphi\wedge \psi|=|\varphi\vee \psi|=|\varphi|+|\psi|+1$

\item $|\pd\varphi|=|\nd \varphi|
=
|\forall\varphi|=|\exists \varphi|=
|\mu p.\varphi|=|\nu p.\varphi|=
|\varphi|+1$.

\end{itemize}
\end{definition}

Sublanguages of $\lang_{\pd \forall }^{\mu}$ are denoted by omitting some of the operators, with the convention that whenever an operator is included, so is its dual (for example, $\lang_\pd$ includes the modalities $\pd,\nd$ but does not allow $\forall,\exists,\mu$ or $\nu$).

\subsection{Convergence spaces}

Spatial interpretations of modal logics are usually presented in terms of topological spaces. Here we will follow an unorthodox route, instead introducing {\em modal spaces,} and as a special case {\em convergence spaces;} both the closure and limit point operators give rise to convergence spaces, while Kripke models can be seen as modal spaces. As a general convention, structures (e.g.~frames or models) will be denoted $\frm A,\frm B,\ldots$, while classes of structures will be denoted $\cls{A}$, $\cls{B}$, $\ldots$. The domain of a structure $\frm A$ will be denoted by $\dom{\frm A}$.

\begin{definition}\label{DefModalSpace}
A {\em normal operator} on a set $A$ is any function ${\oper} \colon \power A\to \power A$ satisfying
\begin{enumerate}[label=({\sc n\arabic*})]

\item $\oper(\varnothing) = \varnothing$, and

\item $\oper (X \cup Y) = \oper (X) \cup \oper  (Y)$.

\end{enumerate}
A {\em modal space} is a pair $\frm A=(\dom{\frm A}, \oper_\frm A)$, where $\dom{\frm A}$ is any set and $\oper_\frm A\colon 2^{\dom{\frm A}} \to 2^{\dom{\frm A}}$ is a normal operator.
If $\rgnX \subseteq \dom{\frm A}$, define $\hat \oper_\frm A (\rgnX)=\overline {\oper_\frm A (\overline  {\rgnX})}$.
We say that $\oper_\frm A$ is a {\em convergence operator} if it also satisfies
\begin{enumerate}[label=({\sc c\arabic*})]
\item $\oper_\frm A(\varnothing)=\varnothing$, and
\item $\oper_\frm A(\oper_\frm A(X))\subseteq \oper_\frm A(Y)$.
\end{enumerate}
If $\oper_\frm A$ is a convergence operator, we will say that $\frm A$ is a {\em convergence space.}

If $X \subseteq \oper_\frm A(X)$ for all $X \subseteq |\frm A|$, we say that $\oper_\frm A$ is {\em inflationary.} An inflationary convergence operator is a {\em closure operator,} and if $\oper_\frm A$ is a closure operator then $\frm A$ is a {\em closure space.}
\end{definition}

This general presentation will allow us to unify semantics over metric spaces with those over arbitrary relational structures.

\begin{remark}
Modal spaces are just {\em neighborhood structures} \cite{PacuitNeighborhood} in disguise; indeed, if $(A,N)$ is a neighborhood structure (i.e., $N\subseteq A\times 2^A$), we may set $\rho_N(X)$ to be the set of $a\in A$ such that every neighborhood of $a$ intersects $X$. Conversely, if $\rgnX \subseteq \dom{\frm A}$ and $a\in \dom{\frm A}$, we let $\rgnX$ be a neighborhood of $a$ if $a\in \hat \oper_\frm A (\rgnX)$.
\end{remark}

As mentioned before, a Kripke frame is simply a structure
 $\frm A =(\dom{\frm A},\rel{\frm A})$, where $\dom{\frm A}$ is a set and 
$\rel{\frm A} \subseteq \dom{\frm A} \times \dom{\frm A}$ is an arbitrary binary relation.
We will  implicitly identify $\frm A$ with the modal space $(|\frm A|,\oper_{\frm A})$, 
where $\oper_\frm A (X) = R^{-1}_\frm A [X]$ for any $X\subseteq |\frm A|$. 
It is readily verified that the operator $\oper_\frm A$ thus defined is normal. In this sense, modal spaces generalize Kripke frames, but
 the structure $(\dom{\frm A},\oper_\frm A)$ is not always a convergence space.
 Nevertheless, convergence spaces may be obtained by restricting our attention to frames where $R_\frm A$ is transitive.

\begin{definition}
Define $\cls K$ to be the class of all Kripke frames, $\cls{K4}$ to be the class of all Kripke frames with a transitive
 relation, $\cls{KD4}$ to be the class of all Kripke frames with a transitive, serial\footnote{Recall that a relation $R\subseteq W\times W$ is {\em serial} if for all $w\in W$, $R(w) \not =w$.}
 relation, and $\cls{S4}$ to be the class of all Kripke frames with a transitive, reflexive relation.
\end{definition}

 The names for these classes are derived from their corresponding modal logics (see e.g.~\cite{Chagrov1997}).
 The following is then readily verified, and we mention it without proof:

\begin{proposition}
\begin{enumerate}

\item If $ {\frm A} \in \cls{K}$, then $(\dom{\frm A}, R^{-1}_{\frm A}[\cdot] )$ is a modal space;

\item If $ {\frm A} \in \cls{K4}$, then $(\dom{\frm A}, R^{-1}_{\frm A}[\cdot])$ is a convergence space, and

\item If $ {\frm A} \in \cls{S4}$, then $(\dom{\frm A}, R^{-1}_{\frm A}[\cdot] )$ is a closure space.
\end{enumerate}
\end{proposition}

Before defining our semantics, we need to briefly discuss least and greatest fixed points. Let $X$ be a set and $f\colon 2^X\to 2^X$ be monotone; that is, if $A\subseteq B\subseteq X$, then $f(A)\subseteq f(B)$. Say that $A_\ast$ is a {\em fixed point} of $f$ if $f(A_\ast)=A_\ast$. Then, the following is well-known:

\begin{theorem}
If $X$ is any set and $f\colon \power X\to \power X$ is monotone, then
\begin{enumerate}

\item  $f$ has a $\subseteq$-least fixed point, which we denote ${\rm LFP}(f)$, and

\item  $f$ has a $\subseteq$-greatest fixed point, which we denote ${\rm GFP}(f)$.

\end{enumerate}

\end{theorem}

This result is discussed in some detail in the context of the spatial $\mu$-calculus in \citep{Goldblatt2017Spatial}.
With this, we turn our attention from frames to {\em models.}

\subsection{Models and truth definitions}

Formulas of $\lang^\mu_{\pd\forall}$ are interpreted as subsets of a convergence space, but first we need to determine the propositional variables that are true at each point.

\begin{definition}
If $\frm A$ is a modal space, a {\em valuation} on $\frm A$ is a function $\V \colon \dom{\frm A}\to \power{\PV}$ (recall that $\PV$ is the set of atoms). 
A modal space $\frm A$ equipped with a valuation $\V $ is a {\em model}. If $\frm A$ is a convergence space, then $(\frm A, \V)$ is a {\em convergence model.}
\end{definition}

If $X\subseteq \dom{\frm A}$ and $p \in \PV$, we define a new valuation $\V_{[X/p]}$ by setting
\begin{itemize}

\item for $q\not=p$, $q\in \V_{[X/p]}(w)$ if and only if $q\in V(w)$, and

\item $p\in \V_{[X/p]}(w)$ if and only if $w\in X$.

\end{itemize}

Now we are ready to define the semantics for $\lang^\mu_{\pd\forall}$.

\begin{definition}\label{DefSem}
Let $(\frm A,\V)$ be a model.
We define the truth set
\[\val\varphi_V = \{w\in \dom{\frm A} : (\frm A, w) \models \varphi\}\]
by structural induction on $\varphi$.

We first need an auxiliary definition for the cases $\mu p.\varphi$ and $\nu p.\varphi$. Suppose inductively that $\lb \varphi \rb$ has been defined for any valuation $\V'$, and define a function
$\valfun p\varphi \V \colon \power {\dom{\frm A}}\to \power {\dom{\frm A}}$ given by $\valfun p\varphi \V (X) = \val\varphi_{\V_{[X/p]}}.$
With this, we define:
\[
\begin{array}{lclllcl}
\lb p \rb_\V&=&\{w\in \dom{\frm A} : p \in \V(w)\}&
&\lb \overline p \rb_\V&=& \{w\in \dom{\frm A} : p \not \in \V(w)\}\\
\lb\varphi\wedge\psi\rb_\V &=&\lb\varphi\rb_\V\cap  \lb\psi\rb_\V&
&\lb\varphi\vee\psi\rb_\V &=&\lb\varphi\rb _\V \cup \lb\psi\rb_\V\\
\val{\pd\varphi}_\V&=& \oper_\frm A(\val\varphi_\V)&
&\val{\nd\varphi}_\V&=&\hat \oper_\frm A(\val\varphi_\V)\\
\val{\exists\varphi}_\V&=&\text{$X$ if $\val\varphi_\V\not=\varnothing$ else $\varnothing$}&
&\val{\forall\varphi}_\V&=&\text{$X$ if $\val\varphi_\V = X$ else $\varnothing$}\\
\lb \mu p.\varphi \rb_\V &=&{\rm LFP}(\valfun p\varphi \V )&
&\lb \nu p.\varphi \rb_\V &=&{\rm GFP}(\valfun p\varphi \V ).
\end{array}
\]
Given a model $(\frm A,\V)$ and formulas $\varphi,\psi\in\lang^\mu_{\pd\forall}$, we say that $\varphi$ is {\em equivalent to} $\psi$ on $\frm A$ if $\val\varphi_\V =\val\psi_\V$. If $\frm A$ is a modal space, $\varphi,\psi$ are equivalent on $\frm A$ if they are equivalent on any model of the form $(\frm A,\V)$, and if $\cls A$ is a class of structures, we say that $\varphi,\psi$ are equivalent over $\cls A$ if they are equivalent on any element of $\cls A$. When $\frm A$ or $\cls A$ is clear from context, we may write $\varphi\equiv\psi$, and if $\psi=\top$ we say $\varphi$ is {\em valid.}
\end{definition}

It is readily verified that if $\overline p$ does not appear in $\varphi$, it follows that $\valfun p\varphi \V $ is monotone, and hence the above definition is sound.
On occasion, if $\frm M$ is a model with valuation $\V$, we may write $\val\cdot_\frm M$ or even $\val\cdot$ instead of $\val\cdot_\V$.
In the case that $(\frm A,\V)$ is a Kripke model, the semantics we have just defined coincide with the standard relational semantics \citep{Chagrov1997}.
To see this, note that for any formula $\varphi$ and any $w\in \dom{\frm A}$, $w\in \val{\pd\varphi}$ if and only if $w\in \oper_{\frm A}(\val\varphi)$, which means that $w\in \rel{\frm A}^{-1}[\val\varphi]$; i.e., there is $v\in\val\varphi$ such that $w \mathrel \rel{\frm A} v$. Thus, the interpretation of $\pd$ coincides with the standard relational interpretation in modal logic.

\section{The extended spatial language}\label{SecExtLan}

The spatial $\mu$-calculus, as we have presented it, may be naturally extended to include other definable operations.
Of course such extensions will not add any expressive power to our language, but as we will see later in the text, they can yield considerable gains in terms of succinctness.
We begin by discussing the closure operator.

\subsection{The closure operator}\label{SubsecClos}

As we have mentioned, the closure operator is definable in terms of the limit point operator on metric spaces. Let us make this precise.
We will denote the closure operator by $\ps\varphi$, defined as a shorthand for $\varphi\vee \pd\varphi$. Dually, the interior operator $\nc\varphi$ will be defined as $\varphi\wedge \nd \varphi$. To do this, let $\lang ^\mu_{\ps\pd\forall}$ be the extension of $\lang^\mu_{\pd\forall}$ which includes $\ps,\nc$ as primitives. Then, for $\varphi \in \lang ^\mu_{\ps\pd\forall}$, define a formula\footnote{We use the general convention that the symbol being replaced by a translation is placed as a subindex, and the symbol used to replace it is used as a superindex. However, this convention is only orientative.} $\trcl\varphi\in \lang ^\mu_{ \pd\forall}$ by letting $\trcl\cdot$ commute with Booleans and all modalities except for $\ps,\nc$, in which case $\trcl{\ps\varphi}= \trcl\varphi \vee \pd  \trcl{\varphi }$ and $\trcl{\nc\varphi} = \trcl \varphi \wedge \nd \trcl {\varphi } $.

Semantics for $\lang ^\mu_{\ps\pd\forall}$ are defined by setting $\val\varphi_\V = \val{\trcl\varphi}_\V$, and we extend Definition \ref{defComplexity} to $\lang^\mu_{\ps\pd\forall}$ in the obvious way, by
\[|\ps\varphi| = |\nc\varphi| = |\varphi| + 1.\]
With this, we can give an easy upper bound on the translation $t^\pd_\ps$.

\begin{lemma}
If $\varphi\in \lang ^\mu_{\ps\pd\forall}$, then $\left | \trcl \varphi \right | \leq 2^{|\varphi|}$.
\end{lemma}

However, this bound is not optimal; it can be improved if we instead define $\pd$ in terms of $\mu$. Define $\trpsmu \colon \lang^\mu_{\ps\pd\forall} \to \lang^\mu_{\pd\forall}$ by replacing every occurrence of $\ps\varphi$ recursively by $\mu p.(\trpsmu(\varphi) \vee \pd p)$ and every occurrence of $\nc\varphi$ recursively by $\nu p.(\trpsmu(\varphi) \wedge \nd p)$ (where $p$ is always a fresh variable), and commuting with Booleans and other operators. Then, we obtain the following:

\begin{lemma}\label{LemmMuShort}
For all $\varphi\in \lang^\mu_{\ps\pd\forall}$, we have that $\varphi \equiv \trpsmu(\varphi)$ over the class of convergence spaces, and $\left |\trpsmu(\varphi)\right| \leq 4|\varphi|$.
\end{lemma}

We omit the proof, which is straightforward. Whenever $\pd$ is interpreted as a convergence operator, $\ps$ is then interpreted as a closure operator. To be precise, given a modal space $\frm A$, define a new operator $\oper^+_\frm A$ on $|\frm A|$ by $\oper^+_\frm A(X)=X\cup \oper_\frm A(X)$. Then, $\frm A^+ = (|\frm A|,\oper^+_\frm A)$ is an inflationary modal space, and if $\frm A$ is a convergence space, it follows that $\frm A^+$ is a closure space.
If $(\frm A,\V)$ is any model and $\varphi$ is any formula, it is straightforward to check that $\val {\ps\varphi}_\V= \oper_\frm A^+(\val\varphi_\V)$. In the setting of Kripke models, we see that $w\in \val{\ps\varphi}_\V$ if and only if $w\in \val\varphi_\V$, or there is $v\in W$ such that $w \rel{\frm A} v$ and $v\in\val\varphi_\V$; as was the case for $\pd\varphi$, this coincides with the standard relational semantics, but with respect to the reflexive closure of $\rel{\frm A}$.

\subsection{Tangled limits}

There is one final extension to our language of interest to us; namely, the tangled limit operator, known to be expressively equivalent to the $\mu$-calculus over the class of convergence spaces, but with arguably simpler syntax and semantics.

\begin{definition}
Let $\lang^{\mu\pt}_{\ps\pd\forall}$ be the extension of $\lang^{\mu}_{\ps\pd\forall}$ such that, if $\varphi_1,\hdots,\varphi_n$ are formulas, then so are $\pt\{\varphi_1,\hdots,\varphi_n\}$ and $\nt\{\varphi_1,\hdots,\varphi_n\}$. We define $\trtan\varphi$ to commute with all operators except $\pt,\nt$, in which case
\begin{align*}
\trtan{\pt\{\varphi_1,\hdots,\varphi_n\}}&=\mu p.\bigwedge_{i\leq n}\pd (p\wedge \trtan{\varphi_i})\\
\trtan{\nt\{\varphi_1,\hdots,\varphi_n\}}&=\nu p.\bigvee_{i\leq n}\nd (p\vee \trtan{\varphi_i}).
\end{align*}
We extend $|\cdot|$ and $\val\cdot$ to $\lang^{\mu\pt}_{\ps\pd\forall}$ by defining
\[|\pt \{\varphi_1,\hdots,\varphi_n\}| = |\nt \{\varphi_1,\hdots,\varphi_n\}| = |\varphi_1| + \hdots + |\varphi_n| +1\]
and $\val\varphi_\V=\val{\trtan\varphi}_\V$.
\end{definition}

We call $\pt$ the {\em tangled limit operator;} this was introduced by Dawar and Otto \cite{DawarOtto} in the context of $\cls{K4}$ frames, then extended by the first author \cite{FernandezIJCAI} to closure spaces and by Goldblatt and Hodkinson \cite{Goldblatt2017Spatial} to other convergence spaces.
For clarity, let us give a direct definition of $\pt$ without translating into the $\mu$-calculus.

\begin{lemma}
If $(\frm A,\V)$ is a convergence model, $\varphi_1,\hdots,\varphi_n$ any sequence of formulas, and $x\in \dom{\frm A}$, then $x\in \val{\pt\{\varphi_1,\hdots,\varphi_n\}}_\V$ if and only if there is $S\subseteq \dom{\frm A}$ such that $x\in S$ and, for all $i\leq n$, $S\subseteq \oper_\frm A(S\cap \val {\varphi_i}_\V)$.
\end{lemma}

Although Dawar and Otto proved in \cite{DawarOtto} that the tangled limit operator is equally expressive as the $\mu$-calculus, they use model-theoretic techniques that do not provide an explicit translation.
As such, we do not provide an upper bound in the following result.

\begin{theorem}\label{TheoDawarOtto}
There exists a function $\trmutan\colon \lang^\mu_{\pd} \to \lang^{\pt}_{\pd}$ such that, for all $\varphi\in \lang^\mu_{\pd} $, $\varphi \equiv \trmutan (\varphi)$ on the class of $\cls{K4}$ models.
\end{theorem}
















Spatial interpretations of the tangled closure and limit operators have gathered attention in recent years (see e.g.~\cite{FernandezTangledDynamics,FernandezNonFinite,GoldblattAiML,GoldblattStudiaLogica}). Later we will show that the tangled limit operator, despite being equally expressive, is exponentially less succinct than the $\mu$-calculus.

\section{Truth-preserving transformations}\label{SecTruthPres}

Let us review some notions from the model theory of modal logics, and lift them to the setting of convergence spaces. We begin by discussing bisimulations, the standard notion of equivalence between Kripke models; or, more precisely between {\em pointed models,} which are pairs $(\frm A, a)$ such that $\frm A$ is a model respectively and $a\in \dom{\frm A}$.  

\subsection{Bisimulations}

The well-known notion of {\em bisimulation} between Kripke models readily generalizes to the setting of convergence spaces, using what we call {\em confluent} relations. Below, we say that two pointed models $({\frm A}, a)$ and $({\frm B}, b)$ {\em differ} on the truth of a propositional variable $p$
when we have $({\frm A}, a) \models p$ whereas $({\frm B}, b)\models \overline p$, or vice-versa. If $({\frm A}, a)$ and $({\frm B}, b)$ do not differ on $p$, then they {\em agree} on $p$.

\begin{definition}\label{DerSimNbh}
Let $\frm A=(\dom{\frm A},\oper_\frm A)$ and $\frm B=(\dom{\frm B},\oper_\frm B)$ be modal spaces and $\chi\subseteq \dom{\frm A}\times \dom{\frm B}$. We say that $\chi$ is {\em forward confluent} if, for all $X\subseteq \dom{\frm A}$,
\[\chi[\oper_\frm A(X)] \subseteq \oper_\frm B (\chi[X]).\]
Say that $\chi$ is {\em backward confluent} if $\chi^{-1}$ is forward confluent, and {\em confluent} if it is forward and backward confluent.

Let $Q \subseteq \PV$ be a set of atoms. If $(\frm A,V_\frm A)$ and $(\frm B,V_\frm B)$ are models, a {\em bisimulation relative to $Q$} is a confluent relation $\chi \subseteq \dom{\frm A}\times \dom{\frm B}$ such that if $a\mathrel \chi b$, then $(\frm A,a)$ agrees with $(\frm B,b)$ on all atoms of $Q$. If $Q$ is not specified, we assume that $Q = \PV$.

A bisimulation between pointed models $(\frm A, a)$ and $(\frm B, b)$ is a bisimulation $\chi\subseteq \dom{\frm A}\times \dom{\frm B}$ such that $a\mathrel \chi b$.
We say that $(\frm A, a)$ and $(\frm B, b)$ are {\em locally bisimilar} if there exists a bisimulation between them, in which case we write $(\frm A, a) \bis (\frm B, b)$. They are {\em globally bisimilar} if there exists a total, surjective bisimulation between them.
\end{definition}

The following is readily verified by a structural induction on formulas (see e.g.~\cite{Goldblatt2017Spatial}). Recall that a variable $p$ is {\em free} if it appears outside of the scope of $\mu p$ or $\nu p$.

\begin{lemma}\label{LemmTruthBisim}
Let $(\frm A,a)$ and $(\frm B,b)$ be pointed models. Then:
\begin{enumerate}

\item if $(\frm A,a)$ and $(\frm B,b)$ are locally bisimilar relative to $Q\subseteq \PV$, then for every $\varphi\in \lang^{\mu\pt}_{\ps\pd}$ all of whose free atoms appear in $Q$, $(\frm A,a)\models\varphi$ if and only if $(\frm B,b)\models\varphi$, and

\item if moreover $(\frm A,a)$ and $(\frm B,b)$ are globally bisimilar relative to $Q\subseteq \PV$, then for every $\varphi\in \lang^{\mu\pt}_{\ps\pd\forall}$ all of whose free atoms appear in $Q$, $(\frm A,a)\models\varphi$ if and only if $(\frm B,b)\models\varphi$.

\end{enumerate}
\end{lemma}

As a special case, we can view an {\em isomorphism} as a bisimulation that is also a bijection. Isomorphism between structures will be denoted by $\cong$.
Similarly, if $\dom{\frm A}=\dom{\frm B}$ and $Q$ is a set of atoms, we say that $\frm A$, $\frm B$ {\em agree everywhere on all atoms in $Q$} if for every $w\in \dom{\frm A}$, $(\frm A,w)$ agrees with $(\frm B,w)$ on all atoms in $Q$. It is easy to see that if this is the case, then $\frm A$ and $\frm B$ are globally bisimilar relative to $Q$.

It is instructive to compare our notion of confluence to more familiar notions in the literature. We begin with the familiar notion of bisimulations between relational models:

\begin{lemma}
If $\frm A=(\dom{\frm A},R_\frm A)$ and $\frm B=(\dom{\frm B},R_\frm B)$ are Kripke frames, then $\chi \subseteq \dom{\frm A} \times \dom{\frm B}$ is forward-confluent if and only if, whenever $a \mathrel R_\frm A a'$ and $a\mathrel \chi b$, there is $b'\in \dom{\frm B}$ such that $a' \mathrel \chi b'$ and $b \mathrel R_\frm B b'$.
\end{lemma}

On metric spaces, confluent functions are related to continuous and open maps.

\begin{lemma}
Let $\frm A=(\dom{\frm A}, \delta_\frm A)$ and $\frm B=(\dom{\frm B},\delta_\frm B)$ be metric spaces with respective closure operators $c_\frm A$, $c_\frm B$ and limit operators $d_\frm A,d_\frm B$, and let $f \colon \dom{\frm A} \to \dom{\frm B}$. Then,
\begin{enumerate}

\item  $f$ is forward-confluent with respect to $c_\frm A$, $c_\frm B$ if and only if $f$ is {\em continuous;} that is, for every $a\in \dom{\frm A}$ and every $\varepsilon >0$, there exists $\eta > 0$ such that if $\delta_\frm A(a, a' ) <\eta$, then $\delta_\frm B(f(a), f(a') ) <\varepsilon$.

\item  $f$ is forward-confluent with respect to $d_\frm A$, $d_\frm B$ if and only if $f$ is continuous and {\em pointwise discrete;} that is, if $a\in\dom{\frm A}$, then there is $\varepsilon>0$ such that if $\delta_\frm A(a,a')<\varepsilon$ and $f(a)=f(a')$, then $a=a'$.

\item $f$ is backward-confluent with respect to $c_\frm A$ and $c_\frm B$ or, equivalently, with respect to $d_\frm A$ and $d_\frm B$, if and only if $f$ is {\em open;} that is, for every $a\in \dom{\frm A}$ and every $\varepsilon >0$, there exists $\eta > 0$ such that if $\delta_\frm B(f(a), b' ) <\eta$, then there is $a'\in \dom{\frm A}$ such that $\delta_\frm A(a, a' ) <\varepsilon$ and $f(a') = b'$.

\end{enumerate}
\end{lemma}

Next, we will review some well-known constructions that yield locally bisimilar models.

\subsection{Generated submodels}

Given Kripke models $\frm A,\frm B$, we say that $\frm A$ is a {\em submodel} of $\frm B$ if $\dom{\frm A} \subseteq \dom{\frm B}$, $\rel{\frm A} = \rel{\frm B}\cap ( \dom{\frm A} \times \dom{\frm A})$, and $V_\frm A(w) = V_\frm B(w)$ for all $w\in \dom{\frm A}$. It is typically not the case that $\frm A$ satisfies the same formulas as $\frm B$, unless we assume that $\dom{\frm A}$ has some additional properties.

\begin{definition}
If $\frm B$ is any Kripke frame or model, a set $U\subseteq \dom{\frm B}$ is {\em persistent} if, whenever $w\in U$ and $w \mathrel \rel{\frm B} v$, it follows that $v\in U$. If $\frm A$ is a subframe (respectively, submodel) of $\frm B$, we say that $\frm A$ is persistent if $\dom{\frm A}$ is.
\end{definition}

In this case, the inclusion $\iota \colon \dom{\frm A} \to \dom{\frm B}$ is a bisimulation, and thus we obtain:

\begin{lemma}
If $\frm A$ is a persistent submodel of $\frm B$ and $w\in \dom{\frm A}$, then $(\frm A,w)$ is locally bisimilar to $(\frm B,w)$.
\end{lemma}

In particular, if we are concerned with satisfiaction of $\lang^{\mu\pt}_{\ps\pd}$-formulas on a pointed model $(\frm B,w)$, it suffices to restrict our attention to the set of points accessible from $w$.

\begin{definition}
Given a binary relation $R$, let $R^\ast$ denote the transitive, reflexive closure of $R$.

Then, given a Kripke frame or model $\frm B$ and $w\in \dom{\frm B}$, we define the {\em generated subframe (respectivel, submodel)} of $w$ to be the substructure of $\frm B$ with domain $R^\ast_\frm B(w)$.
\end{definition}

The following is then obvious from the definitions:

\begin{lemma}
If $\frm B$ is a Kripke structure and $w\in\dom{\frm B}$, then the generated substructure of $w$ is persistent.
\end{lemma}

\begin{remark}
Although we will not need this in the text, persistent substructures can be generalized to other classes of convergence spaces by considering substructures with open domain (see, e.g., \cite{Goldblatt2017Spatial}). However, it is typically not the case that there is a least open substructure containing a given point $w$.
\end{remark}

\subsection{Model amalgamation}\label{SecAmalg}

If $\{A_i : i\in I\}$ is a family of sets, let us use $\coprod_{i\in I} A_i$ to denote its disjoint union in a standard way. We extend this notation to families $\{\frm A_i : i\in I\}$ of Kripke models by setting
\[\coprod_{i\in I}\frm A_i = (\dom{\frm A},R_\frm A,V_\frm A),\]
where
\begin{enumerate}[label=(\roman*)]

\item $\dom{\frm A} = \coprod_{i\in I}\dom{\frm A_i} $,

\item $R_{\frm A} = \coprod_{i\in I} R_{\frm A_i} $, and

\item for $w\in \dom{\frm A}$, $V_{\frm A}(w) = V_{\frm A_i}(w)$ if $w\in \dom{\frm A_i}$.

\end{enumerate}

It is easy to check that, for any $j\in I$, $\frm A_j$ is a persistent substructure of $\coprod_{i\in I}\frm A_i$, and thus we obtain the following from Lemma \ref{LemmTruthBisim}:

\begin{lemma}
If $\{\frm A_i : i\in I\}$ is a family of models, $w\in \dom{\frm A_j}$ and $\varphi\in \lang^{\mu\pt}_{\ps\pd}$, then $(\frm A_j,w)\models \varphi$ if and only if $(\coprod_{i\in I}\frm A_i,w)\models\varphi$.
\end{lemma}

The tools we have presented will be instrumental throughout the text to obtain many of our main results.
Next, we turn our attention to discussing classes of convergence spaces that will be important throughout the text.

\section{Special classes of spaces}\label{SecSpecial}

Our main succinctness results use constructions based on Kripke semantics, which we then `lift' to other classes of spaces. Specifically, we will focus on classes of $\cls{K4}$ models that are confluent images of natural spaces, including Euclidean spaces.
As the latter are connected and confluent maps preserve connectedness, we must work with $\cls{K4}$ frames that share this property.

\subsection{Connectedness}

Given any $\cls {K4}$ frame $\frm A$, there always exists some metric space $\frm X$ such that there is a surjective confluent map $f\colon \dom{\frm X}\to \dom{\frm A}$ \cite{Kudinov2014}. However, if the space $\frm X$ is fixed beforehand, there is not always a guarantee that such a map exists. In particular, this is typically not the case for $\mathbb R^n$ for any $n$, due to the fact that these spaces are {\em connected;} that is, they cannot be partitioned into two disjoint open sets. More formally, if $\mathbb R^n=i(\rgnA)\cup i(\overline {\rgnA})$, then either $\rgnA=\varnothing$ or $\overline {\rgnA}=\varnothing$. This property is characterized by the connectedness axiom
\[\forall (\nc p \vee \nc \overline p)\to (\forall p \vee \forall \overline p),\]
studied by Shehtman in \citep{ShehtmanEverywhereHere}.

Over the class of relational structures the connectedness axiom is not valid in general, but it {\em is} valid over a special class of frames. 

\begin{definition}
Let $\frm A=(\dom{\frm A},R_\frm A)$ be a $\cls{K4}$ frame and $B\subseteq \dom{\frm A}$. We say that $B$ is {\em connected} if for all $w,v\in B$, there are $b_1 \mathrel  \hdots b_{n } \in B$ such that $b_0 =w$, $b_n = v$, and for all $i<n$, either $b_i \mathrel R_\frm A b_{i+1}$ or $b_{i+1} \mathrel R_\frm A b_{i}$.

We say that $\frm A$ is {\em connected} if $\dom{\frm A}$ is connected, and that it is {\em locally connected} if for all $a\in \dom{\frm A}$, $R(a)$ is connected. We say that $\frm A$ is {\em totally connected} if it is both connected and locally connected. The class of totally connected $\cls{KD4}$ frames will be denoted $\cls{TC}$.
\end{definition}

A celebrated result of McKinsey and Tarski \cite{McKinsey1944} states that any formula of $\lang_\ps$ satisfiable over an $\cls{S4}$ frame is satisfiable over the real line, or any other {\em crowded}\footnote{A metric space is {\em crowded} if $d_\frm X(\dom{\frm X}) = \dom{\frm X}$; i.e., if $\frm X$ contains no isolated points.} metric space $\frm X$ satisfying some natural properties.
This result has since received several improvements and variations throughout the years (see e.g.~\cite{Rasiowa,BezhGerke,MintsZhang,KremerStrong}).
We present a powerful variant proven in
\cite{Goldblatt2017Spatial}, which states the following.

\begin{theorem}\label{TheoExistsConfMap}
Let $\frm X=(\dom{\frm X},d_\frm X)$ be a crowded metric space equipped with the limit operator, and $\frm A=(\dom{\frm A},R_\frm A)$ be finite $\cls{TC}$ frame.
Then, there exists a surjective, confluent map $f\colon \dom{\frm X} \to \dom{\frm A}$.
\end{theorem}

Putting together Lemma \ref{LemmTruthBisim} and Theorem \ref{TheoExistsConfMap}, we obtain the following.

\begin{corollary}\label{CorRn}
Let $\frm X=(\dom{\frm X},d_\frm X)$ be a crowded metric space equipped with the limit operator, and $\frm A=(\dom{\frm A},R_\frm A,V_\frm A)$ be a finite $\cls{TC}$ model.
Then, there exists a map $f\colon \dom{\frm X} \to \dom{\frm A}$ and a model $\frm M = (\frm X,V_\frm M)$ such that, for all $\varphi\in \lang^{\mu\pt}_{\ps\pd\forall}$,
\begin{equation}\label{EqTruthMetric}
\val\varphi_\frm M = f^{-1}[\val\varphi_\frm A].
\end{equation}
\end{corollary}

\begin{proof}
The map $f$ is the surjective, confluent map provided by Theorem \ref{TheoExistsConfMap}, and the valuation $V_\frm X$ is defined by $p\in V_\frm X(x)$ if and only if $p \in V_\frm A(f(x))$ for $p\in \PV$ and $x\in \dom{\frm X}$. That \eqref{EqTruthMetric} holds follows from Lemma \ref{LemmTruthBisim}.
\end{proof}

\subsection{Scattered spaces}\label{SecScattered}

In {\em provability logic} \citep{logicofprovability}, a seemingly unrelated application of modal logic, $\nd\varphi$ is interpreted as  `$\varphi$ is a theorem of (say) Peano arithmetic'. Surprisingly, the valid formulas under this interpretation are exactly the valid formulas over the class of scattered limit spaces, as shown by Solovay \cite{Solovay}.
This non-trivial link between proof theory and spatial reasoning allows for an additional and unexpected application of the logics we are considering.
For this, let us define scattered spaces in the context of convergence spaces.

\begin{definition}
A convergence space $\frm A$ is {\em scattered} if, for every $X\subseteq \dom{\frm A}$, if $X\subseteq \oper_\frm A(X)$, then $X=\varnothing$.
%
\end{definition}

In other words, if $X\not=\varnothing$, then there is $a \in X\setminus \oper_\frm A(X)$; such a point is an {\em isolated point} of $X$. Examples of scattered spaces are not difficult to construct.

\begin{lemma}
Let $\frm A = (\dom{\frm A}, \rel{\frm A})$ be any $\cls{K4}$ frame. Then, $\frm A$ is a scattered space if and only if $\rel{\frm A}$ is {\em converse-well-founded;} that is, there do not exist infinite sequences
\[a_0 \mathrel \rel{\frm A} a_1 \mathrel \rel{\frm A} a_2 \mathrel \rel{\frm A} \hdots\]
%
%
%
\end{lemma}

In particular, if $\frm A$ is a finite $\cls {K4}$ frame, then $\frm A $ is scattered as a convergence space if and only if $\rel{\frm A}$ is irreflexive. 
The class of frames with a transitive, converse well-founded relation is named $\cls{GL}$ after G\"odel and L\"ob, whose contributions led to the development of provability logic. 

\begin{proposition}\label{PropTangTrivial}
Let $\varphi_1,\hdots,\varphi_n \in \lang^{\mu\pt}_{\ps\pd\forall}$. Then, ${\pt \{ \varphi_1,\hdots,\varphi_n \}} \equiv \bot$ over the class of scattered spaces.
%
%
%
\end{proposition}

\begin{proof}
Assume that $(\frm A,\V)$ is a scattered limit model. By the semantics of $\pt$, consider a set $S$ such that, for $i\leq n$, $S \subseteq \oper_\frm A (S \cap \val{\varphi_i}_\V)$. But, this implies that $S \subseteq \oper_\frm A (S)$, which, since $\frm A$ is scattered, means that $S=\varnothing$. Since $\val{\pt \{ \varphi_1,\hdots,\varphi_n \}}_\V  $ is the union of all such $S$, we conclude that $\val{\pt \{ \varphi_1,\hdots,\varphi_n \}}_\V = \varnothing $.
\end{proof}

From this we immediately obtain the following.

\begin{corollary}\label{CorTanPd}
There exists a function $\trtanpd\colon \lang^{\pt}_{\pd \forall} \to \lang_{\pd \forall}$ such that $\varphi \equiv \trtanpd(\varphi)$ is valid over the class of scattered spaces and $\left | \trtanpd(\varphi) \right | \leq |\varphi|$ for all $\varphi \in \lang^{\pt}_{\ps\forall}$.
\end{corollary}

Theorem \ref{TheoExistsConfMap} has an analogue for a family of `nice' scattered spaces. Below and throughout the text, we use the standard set-theoretic convention that an ordinal is equivalent to the set of ordinals below it, i.e. $\zeta\in \xi$ if and only if $\zeta<\xi$.

\begin{definition}
Given an ordinal $\Lambda$, define $d \colon 2^\Lambda \to 2^\Lambda$ by letting $\xi\in d(X)$ if and only if $X \cap \xi$ is unbounded in $\xi$.
\end{definition}

Recall that addition, multiplication and exponentiation are naturally defined on the ordinal numbers (see, e.g., \cite{HrbacekJech}) and that $\omega$ defines the least infinite ordinal. The following result can be traced back to Abashidze \cite{abashidze1985} and Blass \cite{blass1990}, and is proven in a more general form by Aguilera and the first author in \cite{AguileraFernandez}.

\begin{theorem}
If $\frm A$ is any finite $\cls{GL}$ frame, then there exists an ordinal $\Lambda < \omega^\omega$ and a surjective map $f\colon \Lambda\to \dom{\frm A}$ that is confluent with respect to the limit operator on $\Lambda$.
\end{theorem}

 As before, this readily gives us the following corollary:

\begin{corollary}\label{CorOrd}
Given a finite $\cls{GL}$-model $\frm A=(\dom{\frm A},R_\frm A,V_\frm A)$, there exists an ordinal $\Lambda < \omega^\omega$, a surjective map $f\colon \Lambda\to \dom{\frm A}$, and a model $\frm M = (\Lambda,V_\frm M)$ such that, for all $\varphi\in \lang^{\mu\pt}_{\ps\pd\forall}$,
\[\val\varphi_\frm M = f^{-1}[\val\varphi_\frm A].\]
\end{corollary}

Now that we have settled the classes of structures we are interested in, we discuss the techniques that we will use to establish our main succinctness results.

\section{Model  equivalence games}\label{SecGames}

The limit-point, or set-derivative, operator $\pd$ is strictly more expressive than the closure 
operator $\ps$ \citep{Bezhanishvili2005}.
 Nevertheless, if we consider a formula such as $\varphi = \underbrace{\ps \ps\hdots \ps}_n \top$, we observe that its translation $\trcl{\varphi}$ into $\lang_{\pd}$ is exponential.

In Section \ref{SecSuccCl}, we will show that this exponential blow-up is indeed inevitable.
To be precise, we wish to show that there is no translation $t\colon \lang_{\ps} \to \lang_{\pd}$
for which exists a sub-exponential function $f(x)$
 such that $t(\varphi)\equiv \varphi$ over the class of convergence spaces and $|t(\varphi)|\leq f(|\varphi|)$. 
In view of Theorem \ref{CorTanPd}, to show that $\varphi\not\equiv\psi$ over the class of convergence spaces (or even metric spaces),
 it suffices to exhibit a model $\frm A \in \cls{K4}$ and $a\in\dom{\frm A}$ such that $(\frm A,a) \models \varphi$ but $ (\frm A,a) \not\models \psi$,
or vice-versa. 
We will prove that such $\frm A$ exists whenever $\psi$ is small by using {\em model equivalence games,}
which are based on sets of pointed models.

We will use $\pointed a,\pointed b$,\,\textellipsis{} to denote pointed models. As was the case for non-pointed structures, for a class of pointed models $\cls{A}$ and a formula $\varphi$, we write $\cls{A}\models \varphi$ when 
$ \pointed a \models \varphi$ for all $ \pointed a  \in \cls{A}$, i.e., $\varphi$ is true in any pointed model in $\cls{A}$, and say that the formulas $\varphi$ and $\psi$ are equivalent on 
a class of pointed models $\cls{A}$ when $ \pointed a  \models \varphi $ if and only if $ \pointed a  \models \psi$ for all pointed models $ \pointed a \in \cls{A}$. We can also define an accessibility relation between pointed models.

\begin{definition}
For a pointed model
$ \pointed a = ({\frm A}, a)$, we denote by $\Box \pointed a$ the set $\{({\frm A}, b) : a \mathrel \rel{\frm A} b\}$, i.e., the set of all pointed models that are successors
of the pointed model $\pointed a$ along the relation $\rel{\frm A}$.
\end{definition}

The game described below is essentially the formula-size game from Adler and Immerman \citep{adlerimmerman}
but reformulated slightly to fit our present purposes. The general idea is that we have two competing
players, Hercules and the Hydra. Given a formula $\varphi$ and a
class of pointed  models $\cls M$, Hercules  is trying to show that there is a ``small'' formula $\psi$ in the language 
 $\lang_{\pd}$ that is equivalent
to $\varphi$ on $\cls{M}$,  whereas the Hydra is trying to show that any such 
$\psi$ is ``big''. Of course, what ``small'' and ``big'' mean depends on the context at hand. The players move by adding and labelling nodes on a game-tree, $T$. Although our use of trees is fairly standard, they a prominent role throughout the text, so let us give some basic definitions before setting up the game.

\begin{definition}
For our purposes, a {\em tree} is a pair $(T,\prec)$, where $T$ is any set and $\prec$ a strict partial order such that, if $\eta\in T$, then $\{\zeta\in T : \zeta\prec \eta\}$ is finite and linearly ordered, and $T$ has a minimum element called its {\em root.} We will sometimes notationally identify $(T,\prec)$ with $T$, and write $\peq$ for the reflexive closure of $\prec$.

Maximal elements of $T$ are {\em leaves.} If $\zeta,\eta\in T$, we say that $\eta$ is a {\em daughter} of $\zeta$ if $\zeta \prec \eta$ and there is no $\xi$ such that $\zeta\prec \xi \prec \eta$.
A {\em path (of length $m$)} on $T$ is a sequence $\vec \eta=(\eta_i)_{i\leq m}$ such that $\eta_{i+1}$ is a daughter of $\eta_i$ whenever $i<m$.
\end{definition}

Next, Definition~\ref{defSatisfactionGame} 
gives the precise moves that Hercules and the Hydra may play in the game.

\begin{definition}\label{defSatisfactionGame}  Let $\cls{M}$ be a class of 
pointed  models and $\varphi$ be a formula.  The  {\em($\varphi$, $\cls{M}$) model equivalence game} ($(\varphi, \cls{M})$-{\meg})
 is played by two players, Hercules and the Hydra, according to the following instructions.\\

\noindent{{\sc setting up the playing field.}}  The Hydra initiates the game by choosing two classes of pointed models $\cls{A},\cls{B}\subseteq\cls{M}$ such that 
$\cls{A}\models\varphi$ and $\cls{B}\models \neg \varphi$.\\

After that, the players continue the  $(\varphi, \cls{M})$-{{\meg}}  on the pair 
$( \cls{A},\cls{B})$ by constructing a finite game-tree $T$, in such a way that each node $\eta\in T$ is  labelled
 with a pair  $( \lft (\eta), \rgt(\eta))$ of classes of pointed models and one symbol 
that is either a literal or one from  the set $ \{ \vee,\wedge, \Box,\Diamond, \exists, \forall\}$. 
We will usually write $\lft (\eta) \circ \rgt(\eta) $ instead of $( \lft (\eta), \rgt(\eta) )$.
The pointed models in $\lft(\eta)$ are  called 
{\em the models  on the left}. Similarly, the pointed models in $\rgt(\eta)$ are called  
{\em the models on the right}.

Any leaf $\eta$ 
can be declared either a {\em head} or a {\em stub}. Once $\eta$ has been declared 
a stub, no further moves can be played on it. The construction of the game-tree begins with a root labeled by
$\cls{A} \circ\cls{B}$ that is declared a head. 

Afterwards, the game continues as long as there is at least one head. In each turn, Hercules goes first by choosing a head $\eta$, labeled by $\cls{L} \circ\cls{R} = \lft (\eta) \circ \rgt(\eta)$. He then plays one of the following moves.\\

\noindent {\sc literal-move.} Hercules  chooses a literal $\iota$ such that  $\cls{L}\models\iota$  and  $\cls{R}\not\models  \iota$. 
The node $\eta$ is declared a stub and labelled with the symbol $\iota$.\\

\noindent {\sc $\vee$-move.} Hercules labels $\eta$ with the symbol $\vee$ and chooses two sets  
$\cls{L}_1,\cls{L}_2 \subseteq\cls{L}$ such that
  $\cls{L}=\cls{L}_1 \cup \cls{L}_2$. Two new heads, labeled by $\cls{L}_1\circ\cls{R}$ and $ \cls{L}_2\circ\cls{R}$,
are added to the tree as daughters of $\eta$.\\

\noindent {\sc $\wedge$-move.} Hercules labels $\eta$ with the symbol $\wedge$ and chooses two sets 
$\cls{R}_1, \cls{R}_2 \subseteq\cls{R}$ such that 
$\cls{R}=\cls{R}_1 \cup \cls{R}_2$. Two new heads, labeled by $\cls{L}\circ\cls{R}_1$ and $\cls{L}\circ\cls{R}_2$,
are added to the tree as daughters of $\eta$.\\
  
\noindent {\sc $\Diamond$-move.} Hercules labels $\eta$ with the symbol $\Diamond$ and, for each pointed 
model $\pointed l \in \cls{L}$,  he chooses a pointed model from $\Box \pointed l$
(if for some  $\pointed l \in \cls{L}$ we have $\Box \pointed l=\emptyset$,
 Hercules cannot play this move). All these new pointed models  are collected in 
the set $\cls{L}_1$.  For each pointed model $\pointed r \in \cls{R}$, the Hydra  replies
 by picking a subset of $\Box \pointed r$.\footnote{In particular, if for some $\pointed r$ we have that $\Box \pointed r=\emptyset$, the Hydra does not 
add anything to $\cls{R}_1$ for the pointed model $\pointed r$.}
All the pointed models chosen by the Hydra are 
collected in the class $\cls{R}_1$. 
A new head labeled by $\cls{L}_1\circ \cls{R}_1$  is added as a daughter to $\eta$.\\

\noindent {\sc $\Box$-move.} Hercules labels $\eta$ with the symbol $\Box$ and, for each pointed model $\pointed r \in \cls{R}$,
 he chooses a pointed model from $\Box \pointed r$ (as before, if for some  $\pointed r \in \cls{R}$ we have 
that $\Box \pointed r=\emptyset$, then  Hercules cannot play this move). All these new pointed models  are 
collected in the set $\cls{R}_1$. The Hydra replies by constructing a class of models $\cls{L}_1$  as follows. 
For each $ \pointed l \in \cls{L}$, she picks a subset of $\Box \pointed l$ and collects all these 
pointed models in the set $\cls{L}_1$.
A head labeled by $\cls{L}_1\circ \cls{R}_1$  is added as a daughter 
to $\eta$.\\

The $(\varphi, \cls{M})$-{\meg} game concludes when there are no heads.
Hercules has a winning strategy  in $n$ moves in the $(\varphi, \cls{M})$-{\meg}  iff no matter how the Hydra plays,
the resulting game tree has $n$ nodes and there are no heads; note that we do not count the move
performed by the Hydra when setting up the playing field.
\end{definition}

The relation between the  $(\varphi, \cls{M})$-{\meg} and formula-size is given by the following result.
The essential features of the proof of the  
 next theorem can be found in any one of  ~\citep{ijcai,succinctnessaijournal,hellaaiml}. 
\begin{theorem}\label{thrm: satisfactionGames}  Hercules has a winning strategy in  $n$ moves in the $(\varphi, \cls{M})$-{\meg} iff 
there is a  $\lang_{\pd}$-formula $\psi$ with   $|\psi|\leq n$ that is equivalent to $\varphi$ on $\cls{M}$.
\end{theorem}

Intuitively, a winning strategy for Hercules in the $(\varphi, \cls{M})$-{\meg} 
 is given by the syntax tree $T_{\psi}$ of  any $\lang_{\pd}$-formula   $\psi$
 that is equivalent to $\varphi$ on $\cls{M}$ (note that $\varphi$ 
is not necessarily a modal formula: for example, it can be a first- or second-order formula, or 
a formula from a different modal language). 
Since Hercules is trying to prove that there {\em exists} a small formula, i.e.,
 the number of nodes in its syntax tree is small,  while the Hydra is trying to show that any such formula is  ``big'',
 if both  Hercules  and the Hydra play optimally 
and Hercules has a winning strategy, then the resulting game tree $T$ is the  syntax tree $T_{\psi}$
of the smallest  modal formula $\psi$ that is equivalent to $\varphi$ on $\cls{C}$.
In particular, if $\cls L\circ\cls R$ is the label of the root, we have that $\psi$ must be true in every element of $\cls L$ and false on every element of $\cls R$, which would be impossible if there were two bisimilar pointed models $\pointed l\in \cls L$ and $\pointed r\in \cls R$. More generally, the subformula $\theta$ of $\psi$ corresponding to any node $\eta$ is true on all pointed models of $\lft(\eta)$ and false on all pointed models of $\rgt(\eta)$. It follows that Hercules loses if any node has bisimilar models on the left and right, provided the Hydra plays well.

As for what it means to `play well', note that the Hydra has no incentive to pick less pointed models in her turns, so it suffices to assume that she plays {\em greedily:}

\begin{definition}
We say that Hydra {\em plays greedily} if:
\begin{enumerate}[label=(\alph*)]

\item whenever Hercules makes a $\pd$-move on a head $\eta$, Hydra replies by choosing all of $\nd \pointed b$ for each $\pointed b\in \rgt (\eta)$, and

\item whenever Hercules makes a $\nd$-move on a head $\eta$, Hydra replies by choosing all of $\nd \pointed a$ for each $\pointed a\in \lft (\eta)$.

\end{enumerate}
\end{definition}

If Hydra plays greedily, Hercules must avoid having bisimilar models on each side:

\begin{lemma}\label{LemBisimLose}
If Hydra plays greedily, no closed game tree contains a node $\eta$ such that there are $\pointed l\in\lft (\eta)$ and $\pointed r \in \rgt (\eta)$ that are locally bisimilar.
\end{lemma}

\proof
We prove by induction on the number of rounds in the game that, once a node $\eta_0$ such that there are bisimilar $ \pointed l \in\lft (\eta_0)$ and $\pointed r \in \rgt (\eta_0)$ is introduced, there will always be a head $\eta$ with bisimilar pointed models on each side. The base case, where $\eta_0$ is first introduced, is trivial, as new nodes are always declared to be heads.

For the inductive step, assume that there are a head $\eta$, $ \pointed l = (\frm L,l) \in\lft (\eta)$, $\pointed r = (\frm R,r) \in \rgt (\eta)$, and a bisimulation $\chi \subseteq \dom{\frm R} \times \dom{\frm L}$ with $l\mathrel \chi r$.
We may also assume that Hercules plays on $\eta$, since otherwise $\eta$ remains on the tree as a head.

Hercules cannot play a literal move on $\eta$ since $l$ and $r$ agree on all atoms. If Hercules plays an $\vee$-move, he chooses $\cls L_1,\cls L_2$ so that $\cls L_1 \cup \cls L_2 = \lft(\eta)$ and creates two new nodes $\eta_1$ and $\eta_2$ labeled by $\cls L_1 \circ \rgt(\eta)$ and $\cls L_2 \circ \rgt(\eta)$, respectively. If $\pointed l \in \cls L_1$, then we observe that $(\pointed l,\pointed r)$ is a pair of bisimilar pointed models that still appears in $\cls L_1 \circ \rgt(\eta)$. If not, then $\pointed l \in \cls L_2$, and the pair appears on $\cls L_2 \circ \rgt(\eta)$. The case for an $\wedge$-move is symmetric.

If Hercules plays a $\pd$-move, then for the pointed model $(\frm L,l') \in \nd \pointed l$ that Hercules chooses, by forward confluence and the assumption that Hydra plays greedily, Hydra will choose at least one pointed model $(\frm R,r')\in \nd \pointed r$ such that $l'\mathrel \chi r'$. Similarly, if Hercules plays a $\nd$-move, then for the pointed model $(\frm R,r') \in \nd \pointed r$ that Hercules chooses, using backwards confluence, Hydra replies by choosing at least one $(\frm L,l')\in \nd \pointed l$ such that $l'\mathrel \chi r'$. It follows that, no matter how Hercules plays, the following state in the game will also contain two bisimilar pointed models, and hence the game-tree will never be closed.
\endproof

\begin{figure}
\begin{center}
\begin{tikzpicture}
[
help/.style={circle,draw=white!100,fill=white!100,thick, inner sep=0pt,minimum size=1mm},
treenode/.style={circle,draw=white!100,fill=white!100, thick, inner sep=0pt, minimum size=5mm},
white/.style={circle,draw=black!100,fill=white!100,thick, inner sep=0pt,minimum size=2mm},
black/.style={circle,draw=black!100,fill=black!100,thick, inner sep=0pt,minimum size=2mm},
xscale=0.75,
yscale=0.4
]


\node at ( 8,28) [treenode](treenode1) []{$\vee$};

\node at ( 6,28.5) [white](m2white) [label=90: $p$]{};
\node at ( 7,28.5) [white](m2blackup) [label=90: $p$]{};
\node at ( 6.5,27.5) [white](m2blackdown) [label=270: $\mathcal{A}_2$, label=180: $\triangleright$]{};

\node at ( 9.5,27.5) [white](nblack) [label=270: $\mathcal{B}$, label=180: $\triangleright$]{};
\node at ( 9,28.5) [white](nleftwhite) []{};
\node at ( 10,28.5) [white](nwhiteright) [label=90: $p$]{};

\node at ( 5,27.5) [white](m1down) [label=270: $\mathcal{A}_1$, label=180: $\triangleright$]{};
\node at ( 4.5,28.5) [white](m1white) []{};
\node at ( 5.5,28.5) [white](m1black2) [label=90: $p$]{};
\node at ( 4.5,29.5) [white](m1up) [label=90: $p$]{};

\begin{scope}[>=stealth, auto]

\draw [->] (m2blackdown) to  (m2blackup);
\draw [->] (m1down) to  (m1black2);
\draw [->] (nblack) to  (nleftwhite);

\draw [->] (m2blackdown) to  (m2white);
\draw [->] (nblack) to  (nwhiteright);
\draw [->] (m1white) to  (m1up);
\draw [->] (m1down) to   (m1white);

\end{scope}

\node at ( 4.5,24) [treenode](treenode2) []{$\Box$};

\node at ( 10,23) [white](2m1down) [label=270: $\mathcal{A}_1$,label=180: $\triangleright$]{};
\node at ( 9.5,24) [white](2m1white) []{};
\node at ( 10.5,24) [white](2m1black) [label=90: $p$]{};
\node at ( 9.5,25) [white](2m1up) [label=90: $p$]{};

\node at ( 6,23) [white](2nblack) [label=270: $\mathcal{B}$,label=180: $\triangleright$]{};
\node at ( 5.5,24) [white](2nleftwhite) []{};
\node at ( 6.5,24) [white](2nwhiteright) [label=90: $p$]{};

\begin{scope}[>=stealth, auto]

\draw [->] (2m1down) to  (2m1white);
\draw [->] (2nblack) to  (2nleftwhite);
\draw [->] (2nblack) to  (2nwhiteright);
\draw [->] (2m1white) to  (2m1up);
\draw [->] (2m1down) to (2m1black);

\end{scope}

\node at ( 11.5,24) [treenode](treenode3) []{$\Diamond$};

\node at ( 13,23) [white](3nblack) [label=270: $\mathcal{B}$, label=180: $\triangleright$]{};
\node at ( 12.5,24) [white](3nleftwhite) [ ]{};
\node at ( 13.5,24) [white](3nwhiteright) [label=90: $p$]{};

\node at ( 3.5,24) [white](3m2white) [label=90: $p$]{};
\node at ( 2.5,24) [white](3m2blackup) [label=90: $p$]{};
\node at ( 3,23) [white](3m2blackdown) [label=270: $\mathcal{A}_2$, label=180: $\triangleright$]{};

\begin{scope}[>=stealth, auto]

\draw [->] (3m2blackdown) to  (3m2blackup);
\draw [->] (3nblack) to  (3nleftwhite);

\draw [->] (3m2blackdown) to  (3m2white);
\draw [->] (3nblack) to  (3nwhiteright);

\end{scope}

\node at (4.5,19.5)[treenode](treenode4)[]{$p$};

\node at ( 10,18.5) [white](4m1down) [label=270: $\mathcal{A}_1$]{};
\node at ( 9.5,19.5) [white](4m1white) [label=180: $\triangleright$]{};
\node at ( 10.5,19.5) [white](4m1black)[label=90: $p$]{};
\node at ( 9.5,20.5) [white](4m1up) [label=90: $p$]{};

\node at ( 6,18.5) [white](4nblack) [label=270: $\mathcal{B}$]{};
\node at ( 5.5,19.5) [white](4nleftwhite) [label=180: $\triangleright$]{};
\node at ( 6.5,19.5) [white](4nwhiteright) [label=90: $p$]{};

\begin{scope}[>=stealth, auto]

\draw [->] (4m1down) to  (4m1white);
\draw [->] (4nblack) to  (4nleftwhite);
\draw [->] (4nblack) to  (4nwhiteright);
\draw [->] (4m1white) to  (4m1up);
\draw [->] (4m1down) to  (4m1black);

\end{scope}

\node at (11.5,19.5)[treenode](treenode5)[]{$\Diamond$};
\node at ( 13,18.5) [white](5nblack1) [label=270: $\mathcal{B}$]{};
\node at ( 12.5,19.5) [white](5nleftwhite1) [label=180: $\triangleright$ ]{};
\node at ( 13.5,19.5) [white](5nwhiteright1) [label=90: $p$]{};

\node at ( 15,18.5) [white](newnblack1) [label=270: $\mathcal{B}$]{};
\node at ( 14.5,19.5) [white](newnleftwhite1) []{};
\node at ( 15.5,19.5) [white](newnwhiteright1) [label=180: $\triangleright$, label=90: $p$]{};

\node at (2.5,19.5) [white](5m3white) [label=90: $p$]{};
\node at ( 3,18.5) [white](5m3blackdown) [label=270: $\mathcal{A}_2$]{};
\node at ( 3.5,19.5) [white](5m3blackup) [label=180: $\triangleright$, label=90: $p$]{};

\node at ( 2,19.5) [white](5m2white) [label=90: $p$]{};
\node at ( 1,19.5) [white](5m2blackup) [label=180: $\triangleright$, label=90: $p$]{};
\node at ( 1.5,18.5) [white](5m2blackdown) [label=270: $\mathcal{A}_2$]{};

\begin{scope}[>=stealth, auto]

\draw [->] (5m3blackdown) to  (5m3white);
\draw [->] (5m2blackdown) to  (5m2blackup);
\draw [->] (5nblack1) to  (5nleftwhite1);
\draw [->] (5m3blackdown) to  (5m3blackup);
\draw [->] (5m2blackdown) to  (5m2white);
\draw [->] (5nblack1) to  (5nwhiteright1);

\draw [->] (newnblack1) to  (newnleftwhite1);
\draw [->] (newnblack1) to  (newnwhiteright1);
\end{scope}

\node at (11.5,15)[treenode](treenode6)[]{$p$};
\node at ( 10,14) [white](6m1down) [label=270: $\mathcal{A}_1$]{};
\node at ( 9.5,15) [white](6m1white) []{};
\node at ( 10.5,15) [white](6m1black) [label=90: $p$]{};
\node at ( 9.5,16) [white](6m1up) [label=180: $\triangleright$, label=90: $p$]{};

\node at (12.5,15)[help](helpempty)[]{$\varnothing$};

\begin{scope}[>=stealth, auto]
\draw [->] (6m1down) to (6m1white);
\draw [->] (6m1white) to  (6m1up);
\draw [->] (6m1down) to  (6m1black);

\end{scope}

\begin{scope}[>=stealth,  thick]

\draw [->] (treenode1) to  (treenode2);
\draw [->] (treenode1) to  (treenode3);
\draw [->] (treenode2) to  (treenode4);
\draw [->] (treenode5) to  (treenode6);
\draw [->] (treenode3) to  (treenode5);
\end{scope}

\end{tikzpicture}

\end{center}
\caption{A closed game tree}
\label{fig:closedtree}
\end{figure}
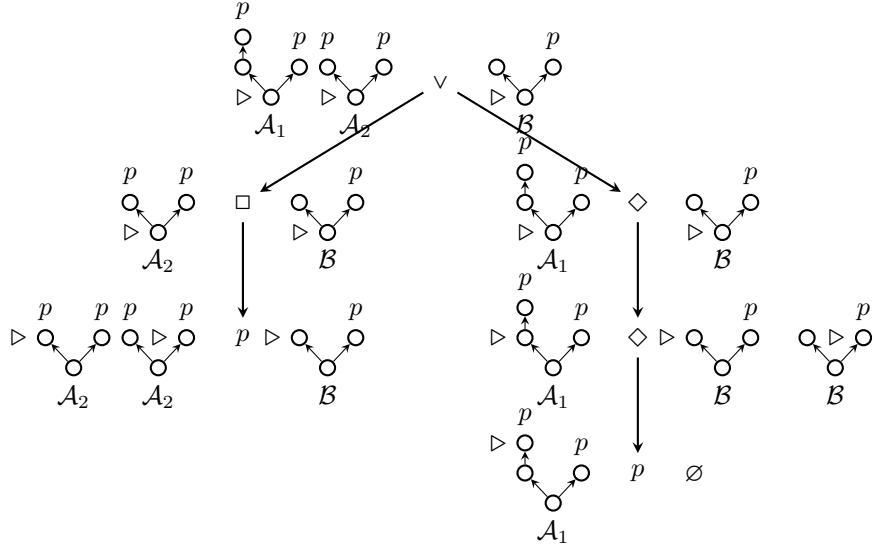

\begin{example}\label{exmp:gametree}
A closed game tree for a model equivalence game is shown in Figure \ref{fig:closedtree}. 
 Pointed models occurring along the nodes of the 
tree are  pairs consisting of the relevant  model $\mathcal{A}_1$, $\mathcal{A}_2$ or $\mathcal{B}$
 and the nodes marked by $\triangleright$. The relations between the points in 
the respective Kripke frame are denoted by the arrows, i.e., if $\frm{F}\in\{\frm{A}_1, \frm{A}_2, \frm{B}\}$,
then, for $w, v \in \dom{\frm{F}}$, we have $w \mathrel \rel{\frm F}v$ if and only if there is an arrow coming out of $w$ and pointing to $v$. 
We have only one proposition $p$ which is shown next to the points in which it is true.
 Note that, if we disregard the Kripke models, the game tree is actually the 
syntax tree of the formula $\nd p \vee \pd\pd p$. It is easily seen that, for any node $\eta$ in the tree, the sub-formula of $\nd p \vee \pd\pd p$
starting at $\eta$ is true in all the pointed models on the left of $\eta$ and false in all the pointed models on the right. 
It is worth pointing out that Hercules could have also won if he had played according to the formula $\nd p \vee \pd(\overline{p}\wedge\pd p)$.
\end{example}

\section{Exponential succinctness of closure over derivative}\label{SecSuccCl}

Now we may use the model equivalence games we have presented to show that the closure operation is exponentially more succinct than the set-derivative operator. Our proof will be based on the following infinite sequence of formulas.

\begin{definition}\label{def:formulas}
For every $n\geq 1$, let the formulas $\varphi_n$ be defined recursively by
\begin{enumerate}[label=(\roman*)]

\item$\varphi_1= \ps p_1$, and

\item $\varphi_{n+1}=\ps(p_{n+1}\wedge \varphi_n)$.

\end{enumerate}
Then, define $\psi_n=\trcl{\varphi_n}.$
\end{definition}

Recall that $t^\pd_\ps$ was defined in Section \ref{SubsecClos}, and that $\phi_n$ is equivalent to $\psi_n$
for every $n$.
Before we proceed, let us give some easy bounds on the size of the formulas we have defined, which can be proven by an easy induction.

\begin{lemma}
For all $n\in\mathbb N$, $|\varphi_n| \leq 3n$ and $|\psi_n| \geq 2^n$.
\end{lemma}

Thus there is an exponential blow-up when passing from $\varphi_n$ to $\psi_n$.
We are going to show that on any class of  models that contains the class of finite $\cls{GL}$ or $\cls{TC}$ models,
we cannot find an  essentially shorter formula than $\psi_n$ in the modal language $\lang_{\pd}$
that is equivalent to $\varphi_n$.
This result will be a consequence of the following.

\begin{theorem}\label{thm:succinctness} 
Let $\cls C$ be either:
\begin{enumerate}[label=(\alph*)]

\item\label{ItTheoSuccA} the class of all finite $\cls{GL}$ frames, or

\item\label{ItTheoSuccB} the class of finite $\cls{TC}$ frames.

\end{enumerate}
Then, for every $n\geq 1$, Hercules has no winning strategy of less than $2^n$ moves in the $(\varphi_n,\cls C)$-\meg.
\end{theorem}

This theorem will be proven later in this section.
Once we do, we will immediately obtain a series of succinctness results:

\begin{proposition}\label{PropSuccinctnessBound}
Let $\cls C$ be a class of convergence spaces containing either:
\begin{enumerate}

\item\label{ItSuccinctnessGL} the class of all finite $\cls{GL}$ frames, or

\item\label{ItSuccinctnessTC} all finite $\cls{TC}$ frames,

\item the set of ordinals $\Lambda<\omega^\omega$, or

\item any crowded metric space $\frm X$.

\end{enumerate}
Then, for all $n\geq 1$, whenever $\psi \in \lang_{\pd}$ is equivalent to $\varphi_n$ over $\cls C$, it follows that $|\psi| \geq 2^{\frac{|\varphi|}3}$.
\end{proposition}

\proof
In the first two cases, the claim follows immediately from Theorems \ref{thm:succinctness} and \ref{thrm: satisfactionGames}.

Now, suppose that $|\psi|\leq 2^{\frac{|\varphi_n|}3}$. Then, by the first claim, there is a finite, pointed $\cls{GL}$ model $(\frm K,w)$ such that $(\frm K,w) \not \models \psi \leftrightarrow \varphi_n$. By Corollary \ref{CorOrd}, there are a model $\frm M$ based on some $\Lambda<\omega^\omega$ and a surjective map $f\colon \Lambda\to\dom{\frm K}$ such that $\val\theta_\frm M = f^{-1}[\val \theta_\frm K]$ for all $\Theta$. In particular, for $\xi \in f^{-1}(w)$ we have that $(\frm M,\xi) \not \models \psi \leftrightarrow \varphi_n$, and thus $\psi$ is not equivalent to $\varphi_n$ on $\omega^\omega$.

If $\cls C$ contains a crowded metric space $\frm X$, we reason analogously, but instead choose $\frm K$ to be a $\cls{TC}$-model and use Corollary \ref{CorRn} to produce the required function $f$.
\endproof

Proposition \ref{PropSuccinctnessBound} will be progressively improved throughout the text until culminating in Theorem \ref{TheoMain}.
In order to  prove  Theorem~\ref{thm:succinctness} when $\cls C$ is the class of finite $\cls{GL}$-models, we are going to  
define, for every $n\geq 1$, a pair of sets of pointed $\cls{GL}$-models $\cls{A}^n$,
$\cls{B}^n$   such that $\cls{A}^n\models \varphi_n$, whereas
$\cls{B}^n\models\neg \varphi_n$.
The Hydra is going to pick $\cls{A}^n$ and  $\cls{B}^n$ when setting up the playing field for 
the $(\varphi_n, \cls{GL})$-{\meg}.
After that, we show that Hercules has no winning strategy of
less than $2^n$ moves.

\subsection{The sets of models $\cls{A}^n$ and $\cls{B}^n$}

Each model in $\cls{A}^n\cup \cls{B}^n$ is based on a finite, transitive, irreflexive tree 
(i.e., a finite, tree-like $\cls{GL}$ model), as illustrated in Figures~\ref{fig:succinctnessModels} and~\ref{fig:succinctnessModelsN}.
The `critical' part of each model lies in its right-most branch as shown in the figures.
To formalize this, let us begin with a few basic definitions.

\begin{definition}
A model $\frm A$ is {\em rooted} if there is a unique $w_0 \in \dom{\frm A}$ with $\dom{\frm A}=\{w_0\} \cup R_\frm A(w_0)$, and {\em tree-like} if $(\dom{\frm A},R_\frm A)$ is a tree. A {\em model with successors} is a model $\frm A$ equipped with a partial function $S_\frm A\colon \dom{\frm A} \pto \dom{\frm A}$ such that $S_\frm A(a)$ is always a daughter of $a$.

If $\frm A$ is a rooted model with successors, the {\em critical branch of $\frm A$} is the maximal path $\vec w=(w_i)_{i\leq m}$ such that $w_0$ is the root of $\frm A$ and $w_{i+1} =S_\frm A(w_i)$ for all $i<m$; we say that $m$ is the {\em critical height} of $\frm A$.

We denote the generated subframe of $S_\frm A(w_0)$ by $S [\frm A]$. If $w\in \dom{\frm A}$ and $S_\frm A(w)$ is defined, we define $S[(\frm A,w)]=(\frm A,S_\frm A(w))$; note that in this case, we do not restrict the domain.
\end{definition}

The partial function $S_\frm A$ will not be used in the semantics, but it will help us to describe Hercules' strategy.
Let us begin by  defining recursively the two sets of pointed models $\cls{A}^n$ and $\cls{B}^n$ with critical branches containing $2^n$ pointed models each.
The formal definition of $\cls{A}^n$ and $\cls{B}^n$ is as follows.

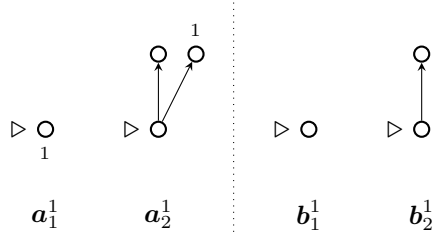
\begin{figure}
\begin{center}
\begin{tikzpicture}
[
help/.style={circle,draw=white!100,fill=white!100,thick, inner sep=0pt,minimum size=1mm},
white/.style={circle,draw=black!100,fill=white!100,thick, inner sep=0pt,minimum size=2mm},
]

 \node at (.5, 2) [white](1A1) [label=270: ${}_1$, label=180: $\triangleright$]{};
\node at (.5, 1.25)[help](1A1)[label=270:$\bm{a}_1^1$]{};
 \node at (2, 2) [white](1A2) [ label=180: $\triangleright$]{};
\node at (2, 1.25)[help](1A2name)[label=270:$\bm{a}_2^1$]{};
 \node at (2.5, 3) [white](2A3) [label=90: ${}_1$]{};
\node at (2, 3) [white](2A2) []{};

\node at ( 3,3.75) [help](helpUpDotted) []{};
\node at ( 3,.5) [help](helpDownDotted) []{};

 \node at (4, 2) [white](1B1) [ label=180: $\triangleright$]{};
\node at (4, 1.25)[help](1B1name)[label=270:$\bm{b}_1^1$]{};
 \node at (5.5, 2) [white](1B2) [ label=180: $\triangleright$]{};
\node at (5.5, 1.25)[help](1B2name)[label=270:$\bm{b}_2^1$]{};
 \node at (5.5, 3) [white](2B2) []{};

\begin{scope}[>=stealth, auto]
\draw [->] (1A2) to  (2A2);
\draw [->] (1A2) to  (2A3);
\draw [->] (1B2) to  (2B2);

\draw [ dotted] (helpUpDotted) to  (helpDownDotted);

\end{scope}

\end{tikzpicture}

\end{center}
\caption{The pointed models in the sets $\cls{A}^1$ and  $\cls{B}^1$. The numbers appearing next to each point $w$ are the indices of the propositional variables true in $w$. Intuitively, ${\frm B}^1_2$ is obtained by ``erasing'' the right branch of ${\frm A}^1_2$.}
\label{fig:succinctnessModels}
\end{figure}

\begin{definition}\label{def:models}
For $n\geq 0$, the $\cls{GL}$-models in the sets $\cls{A}^{n+1}  = \{\frm A^{n+1}_i :i \leq 2^n\}$ and $\cls{B}^{n+1}  = \{\frm B^{n+1}_i :i \leq 2^{n+1}\}$ are defined recursively according to the following cases.
When defined, we will denote the roots of $\frm A^m_j$ and $\frm B^m_j$ by $a^m_j$ and $b^m_j$, respectively, and the pointed models $(\frm A^m_j,a^m_j)$, $(\frm B^m_j,b^m_j)$ by $\pointed a^m_j$ and $\pointed b^m_j$.
\\

\noindent {\sc Case $i\leq 2^{n}$.}
If $n=0$, then $\frm A^1_1$ is a single point $a^1_1$ with $V_{\frm A^1_1}(a^1_1) = \{p_1\}$, and $\frm B^1_1$ is a single point $b^1_1$ with $V_{\frm B^1_1}(b^1_1) = \varnothing$.

If $n>0$, assume inductively that $\cls A^n,\cls B^n$ have been defined. We define $\frm A^{n+1}_i$ to be a copy of $\frm A^{n}_i$, except that the new propositional symbol $p_{n+1}$ is true in the 
root. Similarly, $\frm B^{n+1}_i$ is a copy of $\frm B^{n}_i$, except that $p_{n+1}$ is true in the root.\\

\noindent {\sc Case $i > 2^{n}$.}
Write $i = {2^n+j}$, so that $1\leq j\leq 2^n$. First, set
\[\frm X = \Big ( \coprod_{k=2}^{2^{n}} S[\frm A^{n+1}_k] \Big ) \amalg {\frm B}^{n+1}_j ,\] 
and then construct ${\frm B}^{n+1}_{2^n+j}$ by adding a (fresh) irreflexive root $b^{n+1}_{2^n+j}$ to $\frm X$ which sees all elements of $\dom{\frm X}$ and satisfies no atoms. 
We set
\[S_{{\frm B}^{n+1}_{2^n+j}} =S_{{\frm B}^{n+1}_{j}} \cup \{ (b^{n+1}_{2^n+j} , b^{n+1}_j )\}. \]
The models ${\frm A}^{n+1}_{2^n+j}$ are constructed similarly, except that we take
\[\frm Y = \Big ( \coprod_{k=2}^{2^{n}}  S[ \frm A^{n+1}_k] \Big ) \amalg {\frm B}^{n+1}_j  \amalg {\frm A}^{n+1}_j ,\]
and add an irreflexive root which sees all elements of $\frm Y$ and satisfies no atoms. 
Finally, we set
\[S_{{\frm A}^{n+1}_{2^n+j}} =S_{{\frm A}^{n+1}_{j}} \cup \{ (a^{n+1}_{2^n+j} , a^{n+1}_j )\}. \]
\end{definition}

\begin{example}
Figure \ref{fig:succinctnessModels} shows $\cls A^1$ and $\cls B^1$.
We are using the conventions established in Example~\ref{exmp:gametree} that
each pointed model consists of the relevant model and the point designated with 
$\triangleright$. The indices of the propositional letters that are true on a point are shown next to it.
In any frame $\frm F$ displayed and any $w\in \dom{\frm F}$, $S_\frm F(w)$ is the right-most daughter of $w$ (when it exists).
Note that $\frm A^1_2$ and $\frm B^1_2$ are defined according to the inductive clause, where in the latter $\frm X$ is just $\frm B^1_1$ (as the rest of the disjoint union has empty range) and $\frm Y$ is $\frm B^1_1\amalg \frm A^1_1$.

Next, $\cls A^2$ and $\cls B^2$ are shown in Figure~\ref{fig:succinctnessModels2}, and are obtained as follows. $\frm A^2_1,\frm A^2_2$ are copies of $\frm A^1_1,\frm A^1_2$, but with $p_2$ made true at the root, and $\frm B^2_1,\frm B^2_2$ are defined similarly from $\frm B^1_1,\frm B^1_2$. In this case, we just have that $\coprod_{k=2}^{2^1}S[\frm A^2_k]= S[\frm A^2_1]$, so that for example $\frm B^2_3$ is obtained by taking $\frm X=S[\frm A^2_1 ] \amalg \frm B^2_3$ and $\frm A^2_3$ is obtained by taking $\frm Y=S[\frm A^2_1]\amalg \frm B^2_3 \amalg \frm A^2_3$. Note that we do not take $k=1$ in the disjoint union, as $S[\frm A^1_1]$ is not defined.

Finally, let us show how to obtain the models in $\cls{A}^3$ and $\cls{B}^{3}$ 
with the help of the  models in $\cls{A}^2$ and $\cls{B}^2$ (see Figure \ref{fig:succinctnessModels3}).
Note that the
relations denoted by the arrows are actually transitive but the remaining arrows are not shown, in order to avoid cluttering.  

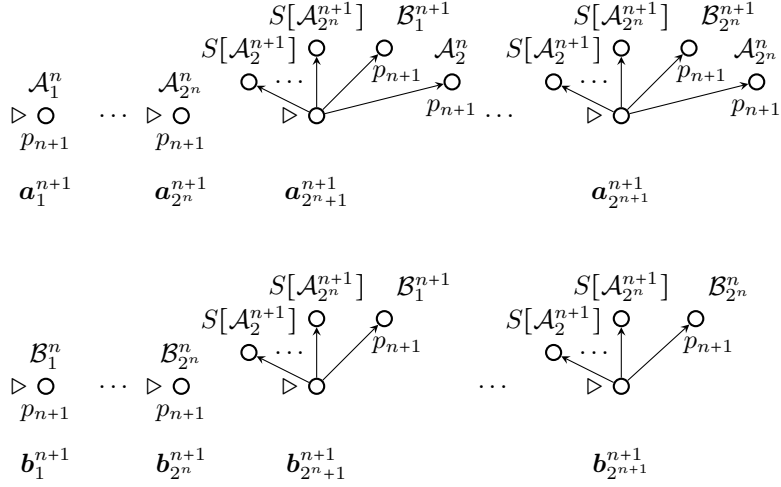
\begin{figure} 
\begin{center}

\begin{tikzpicture}
[
help/.style={circle,draw=white!100,fill=white!100,thick, inner sep=0pt,minimum size=1mm},
white/.style={circle,draw=black!100,fill=white!100,thick, inner sep=0pt,minimum size=2mm},
scale=.9
]

\def\ya{-1}

 \node at (1, 6+\ya) [white](1An) [label=270: $p_{n+1}$,label=90: $\mathcal A^n_1$,label=180:$\triangleright$]{};
\node at ( 1,5.3+\ya) [help](dots1) [label=270:${\bm a}^{n+1}_1$]{};

\node at ( 2,6+\ya) [help](dots1) []{$\ldots$};
 \node at (3, 6+\ya) [white](1An2n) [label=270: $p_{n+1}$,label=90: ${\mathcal A}^n_{2^n}$, label=180:$\triangleright$]{};
\node at ( 3,5.3+\ya) [help](dots1) [label=270:${\bm a}^{n+1}_{2^n}$]{};


 \node at (5, 6+\ya) [white](1An2n1) [label=180:$\triangleright$]{};
\node at ( 5,5.3+\ya) [help](dots1) [label=270:${\bm a}^{n+1}_{2^n+1}$]{};

 \node at (7, 6.5+\ya) [white](1Anup) [label=90: ${\mathcal A}^n_2$,label=270: $p_{n+1}$]{};

\node at (6, 7+\ya) [white](1Bnup) [label=70: $ {\mathcal B}^{n+1}_1$,label=270: $\mbox{  }\mbox{  }\mbox{ }p_{n+1}$]{};

\node at (5, 7+\ya) [white](rest) [label=90:$S{[}{\mathcal A}_{2^n}^{n+1}{]}$]{};

\node at (4.6, 6.5+\ya) [help](dots) []{$\ldots$};

\node at (4, 6.5+\ya) [white](rest1) [label=90:$S{[}{\mathcal A}_2^{n+1}{]}$]{};

\begin{scope}[>=stealth, auto]
\draw [->] (1An2n1) to  (1Anup);
\draw [->] (1An2n1) to  (1Bnup);
\draw [->] (1An2n1) to  (rest);
\draw [->] (1An2n1) to  (rest1);
\end{scope}


\node at (7.7, 6+\ya) [help](dots) []{$\ldots$};

 \node at (9.5, 6+\ya) [white](1An2n2n) [label=180:$\triangleright$]{};
\node at ( 1,0.3) [help](dots1) [label=270:${\bm b}^{n+1}_1$]{};

\node at ( 9.5,5.3+\ya) [help](dots1) [label=270:${\bm a}^{n+1}_{2^{n+1}}$]{};

 \node at (11.5, 6.5+\ya) [white](2nAnup) [label=90: ${\mathcal A}^n_{2^n}$,label=270: $p_{n+1}$]{};

\node at (10.5, 7+\ya) [white](2nBnup) [label=70: ${\mathcal B}^{n+1}_{2^n}$,label=270: $\mbox{ }\mbox{ }\mbox{ }p_{n+1}$]{};

\node at (9.5, 7+\ya) [white](rest11) [label=90:$S{[}\mathcal A^{n+1}_{2^{n}}{]}$]{};

\node at (9.1, 6.5+\ya) [help](dots) []{$\ldots$};

\node at (8.5, 6.5+\ya) [white](rest111) [label=90:$S{[}\mathcal A_2^{n+1}{]}$]{};

\begin{scope}[>=stealth, auto]
\draw [->] (1An2n2n) to  (2nAnup);
\draw [->] (1An2n2n) to  (2nBnup);
\draw [->] (1An2n2n) to  (rest11);
\draw [->] (1An2n2n) to  (rest111);
\end{scope}


\node at (1, 1) [white](1Bn) [label=270: $p_{n+1}$,label=90: ${\mathcal B}^n_1$,label=180:$\triangleright$]{};

\node at ( 2,1) [help](dots1) []{$\ldots$};
\node at ( 3,0.3) [help](dots1) [label=270:${\bm b}^{n+1}_{2^n}$]{};
 \node at (3, 1) [white](1Bn2n) [label=270: $p_{n+1}$,label=90: ${\mathcal B}^n_{2^n}$,label=180:$\triangleright$]{};


 \node at (5, 1) [white](1Bn2n1) [label=180:$\triangleright$]{};
\node at ( 5,0.3) [help](dots1) [label=270:${\bm b}^{n+1}_{2^n+1}$]{};

\node at (6, 2) [white](1BBnup) [label=70: ${\mathcal B}^{n+1}_1$,label=270: $\mbox{ }\mbox{ }\mbox{ }p_{n+1}$]{};

\node at (5, 2) [white](Brest) [label=90:$S{[}{\mathcal A}^{n+1}_{2^{n}}{]}$]{};

\node at (4.6, 1.5) [help](Bdots) []{$\ldots$};

\node at (4, 1.5) [white](Brest1) [label=90:$S{[}{\mathcal A}_2^{n+1}{]}$]{};

\begin{scope}[>=stealth, auto]

\draw [->] (1Bn2n1) to  (1BBnup);
\draw [->] (1Bn2n1) to  (Brest);
\draw [->] (1Bn2n1) to  (Brest1);
\end{scope}


\node at (7.6, 1) [help](Bdots) []{$\ldots$};

 \node at (9.5, 1) [white](1Bn2n2n) [label=180:$\triangleright$]{};
\node at ( 9.5,0.3) [help](dots1) [label=270:${\bm b}^{n+1}_{2^{n+1}}$]{};

\node at (10.6, 2) [white](2nBBnup) [label=70: ${\mathcal B}^n_{2^n}$,label=270: $\mbox{ }\mbox{ }\mbox{ }p_{n+1}$]{};

\node at (9.5, 2) [white](Brest11) [label=90:$S{[}{\mathcal A}_{2^n}^{n+1}{]}$]{};

\node at (9.1, 1.5) [help](Bdots) []{$\ldots$};

\node at (8.5, 1.5) [white](Brest111) [label=90:$S{[}{\mathcal A}_2^{n+1}{]}$]{};

\begin{scope}[>=stealth, auto]

\draw [->] (1Bn2n2n) to  (2nBBnup);
\draw [->] (1Bn2n2n) to  (Brest11);
\draw [->] (1Bn2n2n) to  (Brest111);
\end{scope}

\end{tikzpicture}

\end{center}
\caption{The pointed models  in the sets $\cls{A}^{n+1}$ and  $\cls{B}^{n+1}$.}
\label{fig:succinctnessModelsN}
\end{figure}


As before, the first four pointed models in $\cls{A}^3$  and $\cls{B}^3$ are obtained from the
four models in $\cls{A}^2$ and $\cls{B}^2$, respectively, by simply making the new proposition $p_3$ true in their roots.
To construct the next four models in $\cls{A}^3$ and $\cls{B}^3$, observe that $\coprod_{k=2}^{2^2}S[\frm A^3_k]=S[\frm A^3_2] \amalg S[\frm A^3_3] \amalg S[\frm A^3_4]$. Then, $\frm B^{3}_{4+j}$ is defined by adding a root to $\frm X=S[\frm A^3_2] \amalg S[\frm A^3_3] \amalg S[\frm A^3_4]\amalg \frm B^3_j$, and $\frm A^{3}_{4+j}$ is obtained by adding $\frm A^3_j$ to $\frm B^3_j$.

Finally, we remark that $S[\frm A^3_1] \cong \frm A^1_1$, $S[\frm A^3_2] \cong \frm A^2_1$ and $S[\frm A^3_3]\cong\frm A^2_2$.
\end{example}

In fact, the latter observation is only a special case of a general pattern.

\begin{lemma}\label{LemmSuccessor}
Let $i=2^k+j$, where $j<2^k$ and $n\geq k$ be arbitrary. Then, $S[\frm A^{n+1}_i]$ is isomorphic to $\frm A^{k+1}_j$, and $S[\frm B^{n+1}_i]$ is isomorphic to $\frm B^{k+1}_j$.
\end{lemma}

\proof
By induction on $n$. In the base case, $n=k$, and $\frm A^{k+1}_i$ is defined by the clause for $i>2^k$, from which it is obvious that $S[\frm A^{k+1}_i]\cong \frm A^{k+1}_j$. Similarly, $S[\frm B^{k+1}_i]\cong \frm B^{k+1}_j$.

If $n>k$, then $\frm A^{n+1}_i$ is defined by the clause for $i\leq 2^k$, meaning that it is a copy of $\frm A^{n}_i$ with an atom added to the root. It follows that
\[S[\frm A^{n+1}_i] \cong S[\frm A^{n}_i] \stackrel{\text{\sc ih}}\cong S[\frm A^{k+1}_i],\]
as needed. The claim for $\frm B ^{n+1}_i$ is analogous.
\endproof

In order for the Hydra to be allowed to set up the playing field with $\cls A^n\circ \cls B^n$, we need our formulas $\varphi_n$ to be true on the left and false on the right, which is indeed the case, as we will see next.

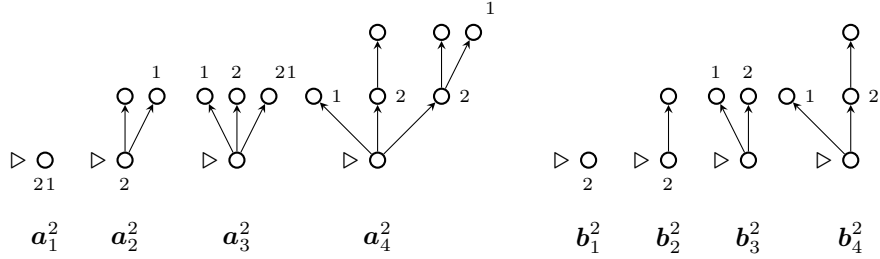
\begin{figure} 
\begin{center}
\begin{tikzpicture}
[
help/.style={circle,draw=white!100,fill=white!100,thick, inner sep=0pt,minimum size=1mm},
white/.style={circle,draw=black!100,fill=white!100,thick, inner sep=0pt,minimum size=2mm},
scale=0.85
]

\def\xa{-.3}
\def\ya{0}
\def\xb{8.5}
\def\yb{5}
\def\xc{-1.4}
\def\xd{-0.5}

 \node at (.5, 7+\ya) [white](1A1) [label=270: ${}_2{}_1$, label=180: $\triangleright$]{};
\node at (.5, 6.25+\ya)[help](1A1)[label=270:$\bm a_1^2$]{};
 \node at (1.75, 7+\ya) [white](1A2) [ label=270: ${}_2$,label=180: $\triangleright$]{};
\node at (1.75, 6.25+\ya)[help](1A2name)[label=270:$\bm a_2^2$]{};
 \node at (2.25, 8+\ya) [white](2A3) [label=90: ${}_1$]{};
\node at (1.75, 8+\ya) [white](2A2) []{};
 \node at (3.5, 7+\ya) [white](3A1) [label=180: $\triangleright$]{};
\node at (3.5, 6.25+\ya)[help](3Aname)[label=270:$\bm a_3^2$]{};
 \node at (4, 8+\ya) [white](3A2) [label=90: $\mbox{ }\mbox{ }\mbox{ }\mbox{ }{}_2{}_1$]{};
\node at (3.5, 8+\ya) [white](3A3) [label=90: ${}_2$]{};
\node at (3, 8+\ya) [white](3A4) [label=90: ${}_1$]{};

 \node at (6+\xa, 7+\ya) [white](4A1) [label=180: $\triangleright$]{};
\node at (6+\xa, 6.25+\ya)[help](4Aname)[label=270:$\bm a_4^2$]{};
 \node at (7+\xa, 8+\ya) [white](4A2) [label=0: ${}_2$]{};
\node at (7.5+\xa, 9+\ya) [white](4A3) [label=90: $\mbox{ }\mbox{ }\mbox{ }\mbox{ }{}_1$]{};
\node at (7+\xa, 9+\ya) [white](4A4) []{};
\node at (6+\xa, 8+\ya) [white](4A5) [label=0: ${}_2$]{};
\node at (6+\xa, 9+\ya) [white](4A7) []{};
\node at (5+\xa, 8+\ya) [white](4A6) [label=0: ${}_1$]{};

 \node at (.5+\xb, 2+\yb) [white](1B1) [label=270: ${}_2$, label=180: $\triangleright$]{};
\node at (.5+\xb, 1.25+\yb)[help](1B1)[label=270:$\bm b_1^2$]{};
 \node at (1.75+\xb, 2+\yb) [white](1B2) [ label=270: ${}_2$,label=180: $\triangleright$]{};
\node at (1.75+\xb, 1.25+\yb)[help](1B2name)[label=270:$\bm b_2^2$]{};
 
\node at (1.75+\xb, 3+\yb) [white](2B2) []{};
 \node at (3.5+\xb+\xd, 2+\yb) [white](3B1) [label=180: $\triangleright$]{};
\node at (3.5+\xb+\xd, 1.25+\yb)[help](3Bname)[label=270:$\bm b_3^2$]{};

\node at (3.5+\xb+\xd, 3+\yb) [white](3B3) [label=90: ${}_2$]{};
\node at (3+\xb+\xd, 3+\yb) [white](3B4) [label=90: ${}_1$]{};

 \node at (6+\xb+\xc, 2+\yb) [white](4B1) [label=180: $\triangleright$]{};
\node at (6+\xb+\xc, 1.25+\yb)[help](4Bname)[label=270:$\bm b_4^2$]{};

\node at (6+\xb+\xc, 3+\yb) [white](4B5) [label=0: ${}_2$]{};
\node at (6+\xb+\xc, 4+\yb) [white](4B7) []{};
\node at (5+\xb+\xc, 3+\yb) [white](4B6) [label=0: ${}_1$]{};

\begin{scope}[>=stealth, auto]
\draw [->] (1A2) to  (2A2);
\draw [->] (1A2) to  (2A3);
\draw [->] (3A1) to  (3A2);
\draw [->] (3A1) to  (3A3);
\draw [->] (3A1) to  (3A4);
\draw [->] (4A1) to  (4A2);
\draw [->] (4A1) to  (4A6);
\draw [->] (4A1) to  (4A5);
\draw [->] (4A2) to  (4A3);
\draw [->] (4A2) to  (4A4);
\draw [->] (4A5) to  (4A7);
\draw [->] (1B2) to  (2B2);

\draw [->] (3B1) to  (3B3);
\draw [->] (3B1) to  (3B4);

\draw [->] (4B1) to  (4B6);
\draw [->] (4B1) to  (4B5);

\draw [->] (4B5) to  (4B7);

\end{scope}

\end{tikzpicture}

\end{center}
\caption{The pointed models  in the sets $\cls{A}^{2}$ and  $\cls{B}^{2}$.}
\label{fig:succinctnessModels2}
\end{figure}

\begin{lemma}\label{LemmPhiTF}
For all $n\geq 1$, $\cls{A}^n\models\varphi_n$, whereas $\cls{B}^n \models \neg \varphi_n$.
\end{lemma}

\proof
By induction on $n$. If $n=1$, then $\varphi_1=\ps p_1$, which is true on all pointed models of $\cls A^1$ and false on all models of $\cls B^1$, as can be seen by inspection on Figure \ref{fig:succinctnessModels}.

If $n>1$, write $n=m+1$. Recall that $\varphi_n = \ps(p_{m+1} \wedge \ps \varphi_{m})$. Fix $i\leq 2^{m+1}$. Let $\pointed a = (\frm A,a) =\pointed a^n_i \in \cls A^n$ and $\pointed b = (\frm B,b)=\pointed b^n_i\in \cls B^n$. We consider two cases.\\

\noindent {\sc Case} $i\leq 2^{m}$. In this case, $\pointed a = \pointed a^{m+1}_i \models p_{m+1}$ by construction. By the induction hypothesis, $\pointed a^m_i\models\varphi_m$. Since $\pointed a $ agrees everywhere with $\pointed a^{m}_i$ on all atoms different from $p_{m+1}$, it follows that $\pointed a \models \varphi_m$ as well, and hence $\pointed a \models p_{m+1}\wedge \ps \varphi_m$, which readily implies that $\pointed a \models \ps (p_{m+1}\wedge \ps \varphi_m)$.

As for $\pointed b = \pointed b^{m+1}_i$, by the induction hypothesis we have that $\pointed b^m_i\not \models \varphi_m$, which implies that $\pointed b \not\models \varphi_m$, since the two agree everywhere on the atoms appearing in $\varphi_m$.
Now, consider arbitrary $v\in \dom{\frm B}$ (so that $b\mathrel R_\frm B v$). If we had that $(\frm B,v) \models \varphi_m$, by the transitivity of $R_\frm B$, we would have that $\pointed b \models \varphi_m$, which is false.
Hence, $\varphi_m$ is false on every point of $\dom{\frm B}$, from which it follows that $\ps \varphi_m$ is false on every point of $\dom{\frm B}$ as well. It follows that $\pointed b\not \models \ps(p_{m+1}\wedge\ps\varphi_m)$.\\

\noindent {\sc Case} $i>2^m$. Write $i=2^m +j$. Since we already have that $\pointed a^m_j \models\varphi_{m+1}$ by the previous case, and $\pointed a^m_j$ is locally bisimilar to $(\frm A ,S_\frm A[a])$, we see that $\pointed a\models\ps\varphi_{m+1}$, which implies that $\pointed a\models\varphi_{m+1}$.

Finally, for $\pointed b$ we see by construction that the only point of $\dom{\frm B}$ that satisfies $p_{m+1}$ is $S_\frm B (b)$. However, by the previous case, $(\frm B, S_\frm B (b))\not\models \varphi_{m+1}$, from which it follows that $(\frm B, S_\frm B (b))\not\models \ps \varphi_{m}$, and hence $\pointed b\not \models \varphi_{m+1}$.
\endproof

\begin{figure} 
\begin{center}

\begin{tikzpicture}
[
help/.style={circle,draw=white!100,fill=white!100,thick, inner sep=0pt,minimum size=1mm},
white/.style={circle,draw=black!100,fill=white!100,thick, inner sep=0pt,minimum size=2mm},
scale=.8
]

\def\xa{-3}
\def\ya{-4.5}
\def\yaaa{-6}
\def\xaa{0}
\def\yaa{0}
\def\xb{-2.2}
\def\yb{\ya}
\def\ybbb{\yaaa}
\def\xbb{0.7}
\def\ybb{0}
\def\xc{-.7}

 \node at (1+\xa, 26+\ya) [white](1A1) [label=270: ${}_3{}_2{}_1$, label=180: $\triangleright$]{};
\node at (0+\xa, 26+\ya)[help](1A1)[label=180:${\bm a}_1^3$]{};
 \node at (1+\xa, 24+\yaaa) [white](1A2) [ label=270: ${}_3{}_2$,label=180: $\triangleright$]{};
\node at (0+\xa, 24+\yaaa)[help](1A2name)[label=180:${\bm a}_2^3$]{};
 \node at (1.5+\xa, 25+\yaaa) [white](2A3) [label=0: ${}_1$]{};
\node at (1+\xa, 25+\yaaa) [white](2A2) []{};
 \node at (1+\xaa, 21.5+\yaa) [white](3A1) [label=180: $\triangleright$,label=270: ${}_3$]{};
\node at (0+\xaa, 21.5+\yaa)[help](3Aname)[label=180:${\bm a}_3^3$]{};
 \node at (1.5+\xaa, 22.5+\yaa) [white](3A2) [label=90: $\mbox{ }\mbox{ }\mbox{ }\mbox{ }{}_2{}_1$]{};
\node at (1+\xaa, 22.5+\yaa) [white](3A3) [label=90: ${}_2$]{};
\node at (0.5+\xaa, 22.5+\yaa) [white](3A4) [label=90: ${}_1$]{};

 \node at (1+\xaa, 18+\yaa) [white](4A1) [label=270:${}_3$,label=180: $\triangleright$]{};
\node at (0+\xaa, 18+\yaa)[help](4Aname)[label=180:${\bm a}_4^3$]{};
 \node at (2+\xaa, 19+\yaa) [white](4A2) [label=0: ${}_2$]{};
\node at (2.25+\xaa, 20+\yaa) [white](4A3) [label=90: $\mbox{ }\mbox{ }\mbox{ }\mbox{ }{}_1$]{};
\node at (1.75+\xaa, 20+\yaa) [white](4A4) []{};
\node at (1+\xaa, 19+\yaa) [white](4A5) [label=0: ${}_2$]{};
\node at (1+\xaa, 20+\yaa) [white](4A7) []{};
\node at (0+\xaa, 19+\yaa) [white](4A6) [label=0: ${}_1$]{};
 \node at (8+\xb, 26+\yb) [white](1B1) [label=270: ${}_3{}_2$, label=180: $\triangleright$]{};
\node at (9+\xb+\xc, 26+\yb)[help](1B1)[label=0:${\bm b}_1^3$]{};
 \node at (8+\xb, 24+\ybbb) [white](1B2) [ label=270: ${}_3{}_2$,label=180: $\triangleright$]{};
\node at (9+\xb+\xc, 24+\ybbb )[help](1B2name)[label=0:${\bm b}_2^3$]{};
 
\node at (8+\xb, 25+\ybbb) [white](2B2) []{};
 \node at (8+\xbb, 21.5+\ybb) [white](3B1) [label=270:${}_3$,label=180: $\triangleright$]{};
\node at (9+\xbb+\xc, 21.5+\ybb )[help](3Bname)[label=0:${\bm b}_3^3$]{};

\node at (8+\xbb, 22.5+\ybb) [white](3B3) [label=90: ${}_2$]{};
\node at (7.5+\xbb, 22.5+\ybb) [white](3B4) [label=90: ${}_1$]{};

 \node at (8+\xbb, 18+\ybb) [white](4B1) [label=270:${}_3$,label=180: $\triangleright$]{};
\node at (9+\xbb+\xc, 18+\ybb )[help](4Bname)[label=0:${\bm b}_4^3$]{};

\node at (8+\xbb, 19+\ybb) [white](4B5) [label=0: ${}_2$]{};
\node at (8+\xbb, 20+\ybb) [white](4B7) []{};
\node at (7+\xbb, 19+\ybb) [white](4B6) [label=0: ${}_1$]{};

\begin{scope}[>=stealth, auto]
\draw [->] (1A2) to  (2A2);
\draw [->] (1A2) to  (2A3);
\draw [->] (3A1) to  (3A2);
\draw [->] (3A1) to  (3A3);
\draw [->] (3A1) to  (3A4);
\draw [->] (4A1) to  (4A2);
\draw [->] (4A1) to  (4A6);
\draw [->] (4A1) to  (4A5);
\draw [->] (4A2) to  (4A3);
\draw [->] (4A2) to  (4A4);
\draw [->] (4A5) to  (4A7);
\putaway{
\draw [->] (5A1) to  (5A2);
\draw [->] (5A1) to  (5A6);
\draw [->] (5A1) to  (5A5);
\draw [->] (5A2) to  (5A3);
\draw [->] (5A2) to  (5A4);
\draw [->] (5A1) to  (5A7);
\draw [->] (5A1) to  (5A8);
}


\draw [->] (1B2) to  (2B2);

\draw [->] (3B1) to  (3B3);
\draw [->] (3B1) to  (3B4);

\draw [->] (4B1) to  (4B6);
\draw [->] (4B1) to  (4B5);

\draw [->] (4B5) to  (4B7);
\end{scope}

\node at (1, 14.5) [white](5A1) [label=180: $\triangleright$]{};
\node at (0, 14.5)[help](5Aname)[label=180:$\bm a_5^3$]{};
 \node at (1, 15.5) [white](5A2) [label=0: ${}_2$]{};
\node at (1.25, 16.5) [white](5A3) [label=90: $\mbox{ }\mbox{ }\mbox{ }\mbox{ }{}_1$]{};
\node at (.75, 16.5) [white](5A4) []{};
\node at (0, 15.5) [white](5A5) [label=90: $\mbox{ }\mbox{ }\mbox{ }{}_2{}_1$]{};
\node at (-1, 15.5) [white](5A6) [label=90: ${}_1$]{};
 \node at (2, 15.5) [white](5A7) [label=90: $\mbox{ }\mbox{ }\mbox{ }{}_3{}_2$]{};
 \node at (3, 15.5) [white](5A8) [label=90: $\mbox{ }\mbox{ }\mbox{ }\mbox{ }\mbox{ }{}_3{}_2{}_1$]{};
\node at (1, 11) [white](6A1) [label=180: $\triangleright$]{};
\node at (0, 11)[help](6Aname)[label=180:$\bm a_6^3$]{};
 \node at (0, 12) [white](6A2) [label=0: ${}_2$]{};
\node at (.25, 13) [white](6A3) [label=90: $\mbox{ }\mbox{ }{}_1$]{};
\node at (-0.25, 13) [white](6A4) []{};
\node at (-1, 12) [white](6A5) [label=90: $\mbox{ }\mbox{ }{}_2{}_1$]{};
\node at (-2, 12) [white](6A6) [label=90: ${}_1$]{};
 \node at (1, 12) [white](6A7) [label=0: ${}_3{}_2$]{};
 \node at (1, 13) [white](6A8) []{};
\node at (2.3, 12) [white](6A9) [label=0: ${}_3{}_2$]{};
\node at (2.3, 13) [white](6A10) []{};
\node at (2.8, 13) [white](6A11) [label=90: ${}_1$]{};
\node at (1, 7.5) [white](7A1) [label=180: $\triangleright$]{};
\node at (0, 7.5)[help](7Aname)[label=180:$\bm a_7^3$]{};
 \node at (0, 8.5) [white](7A2) [label=0: ${}_2$]{};
\node at (.25, 9.5) [white](7A3) [label=90: $\mbox{ }{}_1$]{};
\node at (-0.25, 9.5) [white](7A4) []{};
\node at (-1, 8.5) [white](7A5) [label=90: $\mbox{ }\mbox{ }{}_2{}_1$]{};
\node at (-2, 8.5) [white](7A6) [label=90: $\mbox{ }{}_1$]{};
 \node at (1, 8.5) [white](7A7) [label=0: ${}_3$]{};
 \node at (1.25, 9.5) [white](7A8) [label=90: ${}_2$]{};
\node at (0.75, 9.5) [white](7A9) [label=90: $\mbox{ }{}_1$]{};
\node at (2.3, 8.5) [white](7A10) [label=0: ${}_3$]{};
\node at (2.3, 9.5) [white](7A11) [label=90: ${}_2$]{};
\node at (2.8, 9.5) [white](7A12) [label=90: $\mbox{ }\mbox{ }\mbox{ }\mbox{ }{}_2{}_1$]{};
\node at (1.8, 9.5) [white](7A13) [label=90: ${}_1$]{};
\node at (1, 3) [white](8A1) [label=180: $\triangleright$]{};
\node at (0, 3)[help](8Aname)[label=180:$\bm a_8^3$]{};
 \node at (0, 4) [white](8A2) [label=0: ${}_2$]{};
\node at (.25, 5) [white](8A3) [label=90: $\mbox{ }{}_1$]{};
\node at (-0.25, 5) [white](8A4) []{};
\node at (-1, 4) [white](8A5) [label=90: $\mbox{ }\mbox{ }{}_2{}_1$]{};
\node at (-2, 4) [white](8A6) [label=90: $\mbox{ }{}_1$]{};
 \node at (1, 4) [white](8A7) [label=0: ${}_3$]{};
 \node at (1.25, 5) [white](8A8) [label=0: ${}_2$]{};
\node at (0.75, 5) [white](8A9) [label=90: $\mbox{ }{}_1$]{};
\node at (2.8, 4) [white](8A10) [label=0: ${}_3$]{};
\node at (2.8, 5) [white](8A11) [label=0: ${}_2$]{};
\node at (2.8, 6) [white](8A17) []{};

\node at (3.75, 5) [white](8A12) [label=0: ${}_2$]{};
\node at (4, 6) [white](8A15) [label=0: ${}_1$]{};
\node at (3.5, 6) [white](8A16) []{};
\node at (2.3, 5) [white](8A13) [label=90: ${}_1$]{};
\node at (1.25, 6) [white](8A14) []{};
\node at (8, 14.5) [white](5B1) [label=180: $\triangleright$]{};
\node at (9, 14.5)[help](5Bname)[label=0:$\bm b_5^3$]{};
 \node at (8, 15.5) [white](5B2) [label=0: ${}_2$]{};
\node at (8.25, 16.5) [white](5B3) [label=90: $\mbox{ }\mbox{ }\mbox{ }\mbox{ }{}_1$]{};
\node at (7.75, 16.5) [white](5B4) []{};
\node at (7, 15.5) [white](5B5) [label=90: $\mbox{ }\mbox{ }\mbox{ }{}_2{}_1$]{};
\node at (6, 15.5) [white](5B6) [label=90: ${}_1$]{};
 \node at (9, 15.5) [white](5B7) [label=90: $\mbox{ }\mbox{ }\mbox{ }\mbox{ }{}_3{}_2$]{};
\node at (8, 11) [white](6B1) [label=180: $\triangleright$]{};
\node at (9, 11)[help](6Bname)[label=0:$\bm b_6^3$]{};
 \node at (8, 12) [white](6B2) [label=0: ${}_2$]{};
\node at (8.25, 13) [white](6B3) [label=90: $\mbox{ }\mbox{ }{}_1$]{};
\node at (7.75, 13) [white](6B4) []{};
\node at (7, 12) [white](6B5) [label=90: $\mbox{ }\mbox{ }{}_2{}_1$]{};
\node at (6.3, 12) [white](6B6) [label=90: ${}_1$]{};
 \node at (9, 12) [white](6B7) [label=0: ${}_3{}_2$]{};
 \node at (9, 13) [white](6B8) []{};
\node at (8, 7.5) [white](7B1) [label=180: $\triangleright$]{};
\node at (9, 7.5)[help](7Bname)[label=0:$\bm b_7^3$]{};
 \node at (7, 8.5) [white](7B2) [label=0: ${}_2$]{};
\node at (7.25, 9.5) [white](7B3) [label=90: $\mbox{ }{}_1$]{};
\node at (6.75, 9.5) [white](7B4) []{};
\node at (6, 8.5) [white](7B5) [label=90: $\mbox{ }\mbox{ }{}_2{}_1$]{};
\node at (5, 8.5) [white](7B6) [label=90: $\mbox{ }{}_1$]{};
 \node at (8, 8.5) [white](7B7) [label=0: ${}_3$]{};
 \node at (8.25, 9.5) [white](7B8) [label=90: ${}_2$]{};
\node at (7.75, 9.5) [white](7B9) [label=90: $\mbox{ }{}_1$]{};

\node at (8, 3) [white](8B1) [label=180: $\triangleright$]{};
\node at (9, 3)[help](8Bname)[label=0:$\bm b_8^3$]{};
 \node at (7, 4) [white](8B2) [label=0: ${}_2$]{};
\node at (7.25, 5) [white](8B3) [label=90: $\mbox{ }{}_1$]{};
\node at (6.75, 5) [white](8B4) []{};
\node at (6, 4) [white](8B5) [label=90: $\mbox{ }\mbox{ }{}_2{}_1$]{};
\node at (5, 4) [white](8B6) [label=90: $\mbox{ }{}_1$]{};
 \node at (8, 4) [white](8B7) [label=0: ${}_3$]{};
 \node at (8.25, 5) [white](8B8) [label=0: ${}_2$]{};
\node at (7.75, 5) [white](8B9) [label=90: $\mbox{ }{}_1$]{};
 \node at (8.25, 6) [white](8B10) []{};

\begin{scope}[>=stealth, auto]
\draw [->] (1A2) to  (2A2);
\draw [->] (1A2) to  (2A3);
\draw [->] (3A1) to  (3A2);
\draw [->] (3A1) to  (3A3);
\draw [->] (3A1) to  (3A4);
\draw [->] (4A1) to  (4A2);
\draw [->] (4A1) to  (4A6);
\draw [->] (4A1) to  (4A5);
\draw [->] (4A2) to  (4A3);
\draw [->] (4A2) to  (4A4);
\draw [->] (4A5) to  (4A7);
\draw [->] (5A1) to  (5A2);
\draw [->] (5A1) to  (5A6);
\draw [->] (5A1) to  (5A5);
\draw [->] (5A2) to  (5A3);
\draw [->] (5A2) to  (5A4);
\draw [->] (5A1) to  (5A7);
\draw [->] (5A1) to  (5A8);
\draw [->] (1B2) to  (2B2);

\draw [->] (3B1) to  (3B3);
\draw [->] (3B1) to  (3B4);

\draw [->] (4B1) to  (4B6);
\draw [->] (4B1) to  (4B5);

\draw [->] (4B5) to  (4B7);
\draw [->] (5B1) to  (5B2);
\draw [->] (5B1) to  (5B6);
\draw [->] (5B1) to  (5B5);
\draw [->] (5B2) to  (5B3);
\draw [->] (5B2) to  (5B4);
\draw [->] (5B1) to  (5B7);

\draw [->] (6B1) to  (6B2);
\draw [->] (6B1) to  (6B6);
\draw [->] (6B1) to  (6B5);
\draw [->] (6B2) to  (6B3);
\draw [->] (6B2) to  (6B4);
\draw [->] (6B1) to  (6B7);
\draw [->] (6B7) to  (6B8);
\draw [->] (6A1) to  (6A2);
\draw [->] (6A1) to  (6A6);
\draw [->] (6A1) to  (6A5);
\draw [->] (6A2) to  (6A3);
\draw [->] (6A2) to  (6A4);
\draw [->] (6A1) to  (6A7);
\draw [->] (6A7) to  (6A8);
\draw [->] (6A1) to  (6A9);
\draw [->] (6A9) to  (6A10);
\draw [->] (6A9) to  (6A11);
\draw [->] (7A1) to  (7A2);
\draw [->] (7A1) to  (7A6);
\draw [->] (7A1) to  (7A5);
\draw [->] (7A2) to  (7A3);
\draw [->] (7A2) to  (7A4);
\draw [->] (7A1) to  (7A7);
\draw [->] (7A7) to  (7A8);
\draw [->] (7A7) to  (7A9);
\draw [->] (7A1) to  (7A10);
\draw [->] (7A10) to  (7A11);
\draw [->] (7A10) to  (7A12);
\draw [->] (7A10) to  (7A13);
\draw [->] (8A1) to  (8A2);
\draw [->] (8A1) to  (8A6);
\draw [->] (8A1) to  (8A5);
\draw [->] (8A2) to  (8A3);
\draw [->] (8A2) to  (8A4);
\draw [->] (8A1) to  (8A7);
\draw [->] (8A7) to  (8A8);
\draw [->] (8A7) to  (8A9);
\draw [->] (8A1) to  (8A10);
\draw [->] (8A10) to  (8A11);
\draw [->] (8A10) to  (8A12);
\draw [->] (8A10) to  (8A13);
\draw [->] (8A11) to  (8A17);
\draw [->] (8A8) to  (8A14);
\draw [->] (8A12) to  (8A15);
\draw [->] (8A12) to  (8A16);
\draw [->] (7B1) to  (7B2);
\draw [->] (7B1) to  (7B6);
\draw [->] (7B1) to  (7B5);
\draw [->] (7B2) to  (7B3);
\draw [->] (7B2) to  (7B4);
\draw [->] (7B1) to  (7B7);
\draw [->] (7B7) to  (7B8);
\draw [->] (7B7) to  (7B9);

\draw [->] (8B1) to  (8B2);
\draw [->] (8B1) to  (8B6);
\draw [->] (8B1) to  (8B5);
\draw [->] (8B2) to  (8B3);
\draw [->] (8B2) to  (8B4);
\draw [->] (8B1) to  (8B7);
\draw [->] (8B7) to  (8B8);
\draw [->] (8B7) to  (8B9);
\draw [->] (8B8) to  (8B10);
\end{scope}

\end{tikzpicture}

\end{center}
\caption{The pointed models  in the sets $\cls{A}^{3}$ and  $\cls{B}^{3}$.}
\label{fig:succinctnessModels3}
\end{figure}
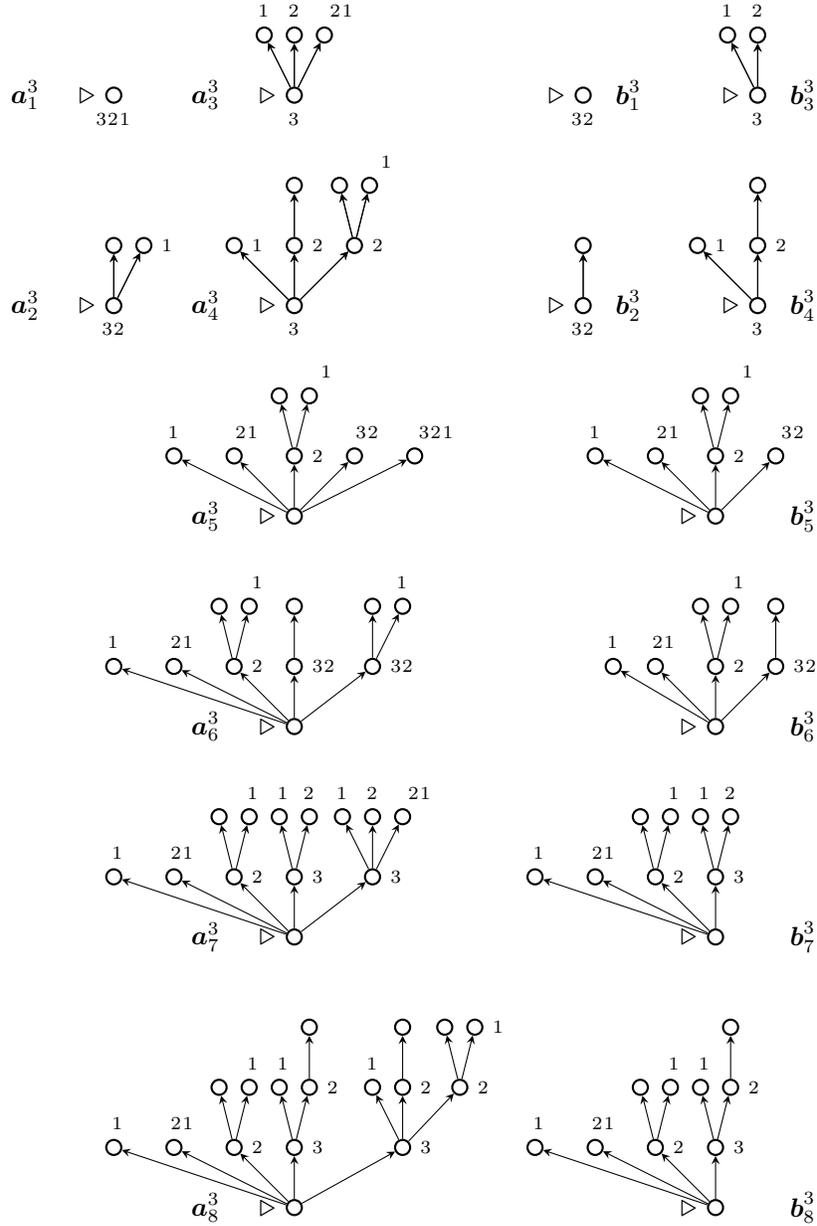

\subsection{The lower bound on the number of moves in  $(\varphi_n, \cls{GL})$-{\meg}}\label{sectScattered}

The pointed models in $\cls A^n\circ \cls B^n$ are constructed so that the critical branch of $\frm A^n_i$ is always very similar to the critical branch of $\frm B^n_i$, differing only at their top point (see, for example, Figure \ref{fig:succinctnessModels3}). Let us make this precise. 

\begin{definition}\label{defSpecialPairs}
Suppose that $\frm M$, $\frm N$ are two finite models with successors and with roots $w$ and $v$, respectively.
We say that $r \in \mathbb N$ {\em distinguishes} $\frm M$ and $\frm N$
if $(\frm M,S_\frm M^r(w))$, $(\frm N,S^r_\frm N(v))$ differ on the truth of a propositional variable, 
but whenever $i<r$, then $(\frm M,S_\frm M^i(w))$, $(\frm N,S^i_\frm N(v))$ agree on the truth of all propositional variables. We call $r$ the {\em distinguishing value} of $\frm M$ and $\frm N$.
\end{definition}

Note that the distinguishing value of two models $\frm M,\frm N$ need not be defined, but when it is, it is unique. Moreover, the distinguishing values of the models we have constructed usually do exist.

\begin{lemma}\label{LemmDistinguishAA}
Fix $n\geq 1$ and $1\leq i < j \leq 2^n$. Suppose that $\frm A^n_i$ and $\frm A^n_j$ have the same critical height $m$. Then, $\frm A^n_i$ and $\frm A^n_j$ are distinguished by some $r<m$, satisfying the following properties:
\begin{enumerate}[label=(\alph*)]

\item\label{ItDZero} If $i\leq 2^{n - 1}$ and $2^{n-1}<j$, then $\frm A^{n}_i$ and $\frm A^{n}_j$ have distinguishing value $0$.

\item\label{ItDOne} If $\frm A^n_i$ and $\frm A^n_j$ have distinguishing value $r$, then $\frm A^{n+1}_i$, $\frm A^{n+1}_j$ have distinguishing value $r$ and $\frm A^{n+1}_{2^n+i}$, $\frm A^{n+1}_{2^n+j}$ have distinguishing value $r+1$.

\end{enumerate}
\end{lemma}

\proof
Assume that $\frm A^{n+1}_i$ and $\frm A^{n+1}_j$ are so that their critical heights have the same value, $m$. To prove \ref{ItDZero}, it suffices to observe that in this case, the root of $\frm A^{n+1}_i$ satisfies $p_{n+1}$, but not the root of $\frm A^{n+1}_j$. For \ref{ItDOne}, if $\frm A^{n}_i$ and $\frm A^{n}_j$ are distinguished by $r$, it is easy to see that $\frm A^{n+1}_i$ and $\frm A^{n+1}_j$ are still distinguished by $r$, since the models agree on all variables $p_k$ with $k\leq n$, and $p_{n+1}$ is true exactly on both roots of the new models. Meanwhile, $S[\frm A^{n+1}_{2^n+i}] \cong \frm A^{n+1}_{i}$ and $S[\frm A^{n+1}_{2^n+j}] \cong \frm A^{n+1}_{j}$, hence $\frm A^{n+1}_i$ and $\frm A^{n+1}_j$ are distinguished by $r+1$ since we have added a new root to each model, both satisfying no atoms.

From this, the existence of a distinguishing value $r<m$ follows by a straightforward induction on $n$. The base case, $n=1$, follows vacuously since no two models in $\cls A^1$ have the same critical height. For the inductive case, we assume the claim for $n$ and prove it for $n+1$. If $i,j\leq 2^n$, then by the induction hypothesis we have that $\frm A^{n}_i$ and $\frm A^{n}_j$ are distinguished by some $r<m$, and by \ref{ItDOne}, $\frm A^{n+1}_i$ and $\frm A^{n+1}_j$ are still distinguished by $r$. If $i\leq 2^n$ and $2^n<j$, by \ref{ItDZero}, $\frm A^{n+1}_i$ and $\frm A^{n+1}_j$ are distinguished by $r=0$. Finally, if $2^n<i,j$, by the induction hypothesis, $\frm A^{n}_{i-2^n}$ and $\frm A^{n}_{j-2^n}$ are distinguished by some $r<m-1$, so that by \ref{ItDOne}, $\frm A^{n+1}_i$ and $\frm A^{n+1}_j$ are distinguished by $r+1<m$.

\endproof

\begin{lemma}\label{LemmDistinguishAB}
Fix $n\geq 1$ and $i \in [1,2^n]$. Then, $\frm A^n_i$ and $\frm B^n_i$ have the same critical height $m$, and are distinguished by $m$.
\end{lemma}

\proof
Proceed by induction on $n$. The base case follows from inspecting $\cls A^1,\cls B^1$, depicted in Figure \ref{fig:succinctnessModels}. For the inductive step, we assume the claim for $n$. If $i\leq 2^n$, since $\pointed A^{n+1}_i$ and $\pointed B^{n+1}_i$ are based on the same frames as $\pointed A^{n}_i$ and $\pointed B^{n}_i$, respectively, it follows from the induction hypothesis that they share the same critical height, $m$. Then, for $k<m$, $S^k[\pointed a^{n}_i]$ and $S^k[\pointed b^{n}_i]$ agree on all atoms by the induction hypothesis, and hence $S^k[\pointed a^{n+1}_i]$ and $S^k[\pointed b^{n+1}_i]$ agree on all atoms too, including $p_{n+1}$, which is true precisely on the roots of both models. Moreover, $S^m[\pointed a^{n}_i]$ and $S^m[\pointed b^{n}_i]$ disagree on some atom (in fact, on $p_1$), hence so too do $S^m[\pointed a^{n+1}_i]$ and $S^m[\pointed b^{n+1}_i]$.

If $2^n<i$, then $S[\pointed a^{n+1}_i] \cong  \pointed a^{n+1}_{i-2^n} $ and $S [\pointed b^{n+1}_i] \cong \pointed b^{n}_{i-2^n}$, which by the previous case share the same critical height, $m$. It follows that $\pointed a^{n+1}_i$ and $\pointed b^{n+1}_i$ also share the same critical height, $m +1$. Moreover, $\pointed a^{n+1}_i$ and $\pointed b^{n+1}_i$ agree on all atoms (they are all false), and since $S^{r+1}[\pointed a^{n+1}_i] \cong S^{r}[\pointed a^{n}_i] $ and $S^{r+1}[\pointed b^{n+1}_i] \cong S^{r}[\pointed b^{n}_i] $ for $r\leq m$, it follows once again by the previous case that they share the same atoms for $r<m$ and disagree on some atom for $r=m$.
\endproof

By {\em twins of height $k$} we mean a pair of the form $(S^k[\pointed a^n_i],S^k[\pointed b^n_i])$, where $i\leq 2^n$ and both expressions are defined.
If $\cls{L}$ is a set of pointed models from $\cls A^n$ and $\cls R$ from $\cls B^n$,
we say that there are twins of height $k$ in 
$\cls{L}\circ\cls{R}$ if there are twins $(S^k[\pointed a^n_i],S^k[\pointed b^n_i])$ such that 
$S^k[\pointed a^n_i] \in \cls{L}$ and $S^k[\pointed b^n_i]\in \cls{R}$.

For example, the pointed models 
$(\pointed a^{n+1}_{2^n+1} , \pointed b^{n+1}_{2^n+1})$ 
shown in Figure~\ref{fig:succinctnessModelsN} are twins of height zero, and $(S[\pointed a^{n+1}_{2^n+1}] , S[\pointed b^{n+1}_{2^n+1}])$ are twins of height one.
Note that the two pairs share the same models, and vary only on the evaluation point: $\pointed a^{n+1}_{2^n+1}$ and $\pointed b^{n+1}_{2^n+1}$ are evaluated at their respective roots, $S[\pointed a^{n+1}_{2^n+1}]$ and $S[\pointed b^{n+1}_{2^n+1}]$ at the rightmost daughters of these roots.

The following lemma tells us that, while models $\frm A^n_i$ and $\frm B^n_j$ can never be bisimilar (since one satisfies $\varphi_n$ but the other does not), they can come quite close to being so.

\begin{lemma}\label{LemmBoxmove}
Fix $n\geq 1$ and $r\geq 0$.
\begin{enumerate}[label=(\roman*)]

\item\label{ItSpecialPairs} If $(\pointed a,\pointed b)$ are twins of height $r$ in $\cls A^n\circ \cls B^n$ and $\pointed a' \in \nd \pointed a$ is such that $\pointed a'\not = S [ \pointed a ]$, it follows that there is $\pointed b'\in \nd\pointed b$ such that $\pointed a' \bis \pointed b'$

\item\label{ItBoxOne}  If $(\pointed a,\pointed b)$ are twins of height $r$ in $\cls A^n\circ \cls B^n$ and $\pointed b' \in \nd \pointed b$, it follows that there is $\pointed a' \in \nd \pointed a$ such that $\pointed a' \bis \pointed b'$.

\item\label{ItBoxTwo} If $1 \leq i<j\leq 2^n$, ${\frm A}^n_i$ and ${\frm B}^n_j$ have the same critical height $m>r$, and they
are distinguished by $r$, then there is $\pointed b'\in \nd S^{r}[{\pointed b}^n_j]$ such that $S^{r+1}[{\pointed a}^n_i] \bis \pointed b'$.

\end{enumerate}

\end{lemma}  

\begin{proof}
The three claims are proven by induction on $n$, and their proofs have a very similar structure.\\

\noindent{\sc Claim} \ref{ItSpecialPairs}. The base case (for $n=1$) follows by observing Figure \ref{fig:succinctnessModels}. Indeed, the only twins are $(\pointed a^1_1,\pointed b^1_1)$, $(\pointed a^1_2,\pointed b^1_2)$ and $(S[\pointed a^1_2],S[\pointed b^1_2])$. Of these, only for $(\pointed a^1_2,\pointed b^1_2)$ can we choose $\pointed a'$ as in the antecedent, so in the other cases the claim is vacuously true. But in this case, we have that $\pointed a'$ must be $\frm A^1_2$ evaluated at the top-left point, which is clearly locally bisimilar to $\frm B^1_2$ evaluated at its top point.

Otherwise, assume the claim for $n$, and let us establish it for $n+1$. Write $\pointed a=(\frm A,a)$ and $\pointed b=(\frm B,b)$, and assume that $i\in[1,2^{n+1}]$ is so that $(\pointed a,\pointed b) = (S^r[\pointed a^{n+1}_i], S^r[ \pointed b^{n+1}_i])$. Let $\pointed a'=(\frm A,a')\in \nd\pointed a$ be such that $\pointed a'\not= S[\pointed a]$, and consider two cases, according to $i$.\\

\noindent {\bf Case $i\leq 2^n$.} Observe that $a'$ cannot be the root of $\frm A$, since $R_\frm A$ is irreflexive. It follows that $(\frm A,a') \bis (\frm A^n_i,a')$ since, by definition, $\frm A^n_i$ and $\frm A=\frm A^{n+1}_i$ disagree only at the root. By the induction hypothesis, there is $b'\in\dom{\frm B^n_i}$ such that $b\mathrel R_\frm B b'$ and $(\frm A^n_i,a') \bis (\frm B^n_i,b')$, but once again $b'$ cannot be the root so we must have that $(\frm B^n_i,b') \bis (\frm B,b') \in \nd \pointed b$, as claimed.\\

\noindent {\bf Case $i> 2^n$.} We consider two sub-cases. First, assume that $a$ is {\em not} the root of $\frm A$, so that $a\in \dom{S[\frm A]}$. In this case, by Lemma \ref{LemmSuccessor}, $(\frm A,a) \bis (\frm A^n_{i-2^n},a)$ and $(\frm B,b) \bis (\frm B^n_{i-2^n},b)$. We can then apply the case for $i\leq 2^n$ to find $b'\in \dom{S[\frm B]}$ such that $ b \mathrel R_\frm B b'$ and $(S[\frm B],b') \bis (S[\frm A],a')$. This gives us that $(\frm B,b') \bis (\frm A,a')$ as well.

Otherwise, assume that $a$ is the root of $\frm A$, so that $b$ is also the root of $\frm B$. If $a' \in \left |\coprod_{j=2}^{2^n} S[\frm A^{n+1}_j]\amalg \frm B^{n+1}_j\right |$, we can take $b'=a'$, and clearly $(\frm A,a')$ is locally bisimilar to $(\frm B,b')$. Otherwise, $a'\in \dom{S[\frm A]}$, and by the assumption $a'\not =  S_\frm A(a)$.

If $a'=S_\frm A(S_\frm A(a))$, we let $b'$ be the copy of $S_\frm A(S_\frm A(a))$ in $\coprod_{j=2}^{2^n} S[\frm A^{n+1}_j]$.
If not, since $S_\frm A(a)$ is the root of $S[\frm A]$, we have that $S_\frm A(a)\mathrel R_\frm A a'$. Since $S[\pointed a] \bis \pointed a^{n+1}_{i-2^n}$, we can apply the case for $i\leq 2^n$ to find $b'\in \dom{S[\frm B]}$ such that $S_\frm B(b)\mathrel R_\frm B b'$ and $(\frm B,b') \bis (\frm A,a')$. By transitivity we also have that $b\mathrel R_\frm B b'$, as needed.\\

\noindent {\sc Claim} \ref{ItBoxOne}. The base case can readily be verified for $\cls A^1\circ \cls B^1$ on Figure \ref{fig:succinctnessModels}.
\putaway{
For the inductive step, assume the claim for $n$; let us prove it for $n+1$. Suppose that $S^r[\pointed a^{n+1}_j]$ and $S^r[\pointed b^{n+1}_j]$ are both defined and $\pointed b = (\frm B^{n+1}_j,b)\in \nd S^r[\pointed b^{n+1}_j]$. If $j\leq 2^n$, then the induction hypothesis tells us that there is $(\frm A^n_j,a) \in \nd S^{r}[\pointed a^n_j]$ that is locally bisimilar to $(\frm B^{n}_j,b)$; but $(\frm A^n_j,a)\bis (\frm A^{n+1}_j,a)$ and $(\frm B^{n}_j,b) \bis (\frm B^{n+1}_j,b)$, as the models differ only at the roots (which $a,b$ are not), yielding the claim.

If $j > 2^n$, consider two cases. If $r=0$, then
\[b \in \left|\coprod^{2^n}_{k=2} S[\frm A^{n+1}_k] \amalg \frm B^{n+1}_{j-2^n} \right | \subseteq \left|\coprod^{2^n}_{k=2} S[\frm A^{n+1}_k] \amalg \frm B^{n+1}_{j-2^n} \amalg \frm A^{n+1}_{j-2^n}\right |,\]
i.e.~$b$ is a point of $\frm A^{n+1}_j$ as well and we can just take $\pointed a= (\frm A^{n+1}_{j},b)$. If instead $r>0$, then $b\in \dom{S[\frm B^{n+1}_j]}$, but $S[\pointed b^{n+1}_j] = \pointed b^{n+1}_{j-2^n}$, and by the previous case we can take $(\frm A^{n+1}_{j-2^n},a) \in \nd S^{r-1}[\pointed b^n_{l-2^n}]$ that is locally bisimilar to $(\frm B^{n+1}_{j-2^n},b)$. But then, $(\frm A^{n+1}_{j-2^n},a ) \bis (\frm A^{n+1}_{j},a ) \in  \nd S^{r}[\pointed a^{n+1}_j]$ and $(\frm B^{n+1}_{j-2^n},b) \bis \pointed b $, so that $(\frm A^{n+1}_{j},a ) \bis \pointed b$, as needed.
}
The inductive step follows the same structure as that for claim \ref{ItSpecialPairs} by swapping the roles of $\pointed a$ and $\pointed b$, except that in the case where $i> 2^n$ and $b$ is the root of $\frm B$, the proof is somewhat simplified as we always have $b' \in \left |\coprod_{j=2}^{2^n} S[\frm A^{n+1}_j]\amalg \frm B^{n+1}_j\right |$, and thus we can always take $a'=b'$.\\

\noindent {\sc Claim} \ref{ItBoxTwo}. It is obvious that the proposition is trivially true for $n=1$  because the respective critical heights of ${\frm A}^1_1$ and ${\frm B}^1_2$ are different.
The inductive step also follows a similar structure as before, but in order to apply the induction hypothesis we must also pay some attention to the distinguishing values.\\

\noindent{\bf Case $i,j\leq 2^n$.} By Lemma \ref{LemmDistinguishAA}\ref{ItDOne}, if $\frm A^{n+1}_i$ and $\frm B^{n+1}_j$ are distinguished by $r$, then so are $\frm A^{n}_i$ and $\frm B^{n}_j$, which by the induction hypothesis tells us that there is $(\frm B^n_j,b') \in \nd S^{r}[\pointed b^n_j]$ that is locally bisimilar to $S^{r+1}[\pointed a^n_i]$. Reasoning as in in the case for $i\leq 2^n$ in claim \ref{ItSpecialPairs}, this yields $S^{r+1}[\pointed a^{n+1}_i ]\bis (\frm B^{n+1}_j,b') \in \nd S^r[\pointed b^{n+1}_j]$.\\

\noindent{\bf Case $i,j > 2^n$.} We once again use Lemma \ref{LemmDistinguishAA}\ref{ItDOne} (twice) to see that if $\frm A^{n+1}_i$ and $\frm B^{n+1}_j$ are distinguished by $r$, then $\frm A^{n+1}_{i-2^n}$ and $\frm B^{n+1}_{j-2^n}$ are distinguished by $r-1$. By the case for $i,j\leq 2^n$, this tells us that there is $(\frm B^{n+1}_{j-2^n},b') \in \nd S^{r-1}[\pointed b^{n+1}_{j-2^n}]$ that is locally bisimilar to $S^{r}[\pointed a^{n+1}_{i-2^n}]$. Setting $\pointed b' = (\frm B^{n+1}_{j},b') $, we reason as in the proof of the case $i>2^n$ in claim \ref{ItSpecialPairs} to obtain $S^{r+1}[\pointed a^{n+1}_i] \bis \pointed b'$ and $\pointed b' \in \nd S^{r}[\pointed b^{n+1}_j]$.\\

\noindent{\bf Case $i \leq  2^n < j$.} By Lemma \ref{LemmDistinguishAA}\ref{ItDZero}, $\frm A^{n+1}_i$ and $\frm B^{n+1}_j$ are distinguished by $r = 0$. But, $\coprod_{k=2}^{2^n} S[\frm A^{n+1}_k]$ already contains a copy of $S[\pointed a^{n+1}_i]$, and we use this copy as $\pointed b'$.
\putaway{
If $j,l\leq 2^n$, then by Lemma \ref{LemmDistinguishAA}\ref{ItDOne}, if $\frm A^{n+1}_j$ and $\frm B^{n+1}_l$ are distinguished by $r$, then so are $\frm A^{n}_j$ and $\frm B^{n}_l$, which by the induction hypothesis tells us that there is $(\frm B^n_l,b) \in \nd S^{r}[\pointed b^n_l]$ that is locally bisimilar to $S^{r+1}[\pointed a^n_j]$, which as in the proof of Claim \ref{ItBoxOne} yields $S^{r+1}[\pointed a^{n+1}_j ]\bis (\frm B^{n+1}_l,b) \in \nd S^r[\pointed b^{n+1}_l]$.

If $j,l > 2^n$, we once again use Lemma \ref{LemmDistinguishAA}\ref{ItDOne} (twice) to see that if $\frm A^{n+1}_j$ and $\frm B^{n+1}_l$ are distinguished by $r$, then $\frm A^{n+1}_{j-2^n}$ and $\frm B^{n+1}_{l-2^n}$ are distinguished by $r-1$, which by the previous case tells us that there is $(\frm B^{n+1}_{l-2^n},b) \in \nd S^{r-1}[\pointed b^{n+1}_{l-2^n}]$ that is locally bisimilar to $S^{r}[\pointed a^{n+1}_{j-2^n}]$. Setting $\pointed b = (\frm B^{n+1}_{l-2^n},b) $, we reason as in the proof of Claim \ref{ItBoxOne} to obtain $S^{r+1}[\pointed a^{n+1}_j] \bis \pointed b$ and $\pointed b \bis (\frm B^{n+1}_l,b) \in \nd S^{r}[\pointed b^{n+1}_l]$, as needed.

Finally, if $j\leq  2^n<l$, we use Lemma \ref{LemmDistinguishAA}\ref{ItDZero} to see that $\frm A^{n+1}_j$ and $\frm B^{n+1}_l$ are distinguished by $r = 0$. But, $\coprod_{k=2}^{2^n} S[\frm A^{n+1}_k]$ already contains a copy of $S[\pointed a^{n+1}_j]$, and we use this copy as $\pointed b$.}
\end{proof}

\begin{example}
Lemma \ref{LemmBoxmove}\ref{ItBoxTwo} is best understood by looking at the models in 
Figures~\ref{fig:succinctnessModels2} and~\ref{fig:succinctnessModels3}.
A simple inspection of
Figure~\ref{fig:succinctnessModels2} is enough to see that 
$\pointed a^2_2$ and $\pointed b^2_3$
differ on the truth of $p_2$ and that $S[{\pointed a}^2_2]$ is locally bisimilar to ${\frm B^2_3}$ at its top-left point.
Meanwhile, in Figure \ref{fig:succinctnessModels3},  $\pointed a^3_2$ and 
$\pointed b^3_3 $ differ on the value of $p_2$, i.e., the smallest number for which the critical branches of ${\frm A}^3_2$ and 
${\frm B}^3_3$ differ on the truth of a propositional variable is zero. Obviously, $S[\pointed a^3_2]$
satisfies only $p_1$ and the same applies to the {\em left} successor point of $\pointed b^3_3$. 
In a similar way, we see that the smallest number for which the critical branches of  ${\frm A}^3_6$ and  ${\frm B}^3_7$
differ on the truth of a propositional variable is one because the rightmost daughters $s$ and $t$ of the roots of  
${\frm A}^3_6$ and ${\frm B}^3_7$, respectively, differ on the truth of $p_2$. Again, we have that the rightmost daughter of $s$ 
satisfies only $p_1$ and the same applies to the left successor of the only node in ${\frm B}^3_7$ that satisfies  $p_3$.
\end{example}

\putaway{
\begin{strategy}\label{StratLB}
For any twins  $({\frm A}^n_j, w)$, $({\frm B}^n_j, v)$ of height $k$ in a given position $\eta$ with
 $\cls{L}\circ\cls{R}$, when Hercules plays a $\pd$-move the Hydra replies with all pointed models from $\nd({\frm B}^n_j, v)$. 
Similarly, if Hercules plays a $\nd$-move, the Hydra replies with all the pointed models in  $\nd({\frm A}^n_j, v)$.\david{Please revise this to a full strategy that covers any move Upper Bound makes, including the bad moves.}\petar{I hope it's better now.}
 \end{strategy}
}

Lemma \ref{LemmBoxmove} shows us that the moves that Hercules can make in order to win are, in fact, rather restricted. Below, for fixed $n\geq 1$, say that Hydra {\em plays amazingly} if she labels the root by $\cls A^n \circ \cls B^n$ and plays greedily.

\begin{lemma}\label{lemm:moves}
Assume that the Hydra plays amazingly.
For any node $\eta$ (not necessarily a leaf) in a closed game tree $T$
 for the $(\varphi_n, \cls{GL})$-{\meg}:

\begin{enumerate}[label=(\alph*)]

\item\label{ItTwinsDiamond} if there are twins $( \pointed a ,  \pointed b) $ in $\lft(\eta)\circ \rgt(\eta)$, then Hercules did not play a $\nd$-move in  $\eta$;

\item\label{ItTwinsBox} if there are twins $( \pointed a ,  \pointed b) $ in $\lft(\eta)\circ \rgt(\eta)$ and Hercules played a $\pd$-move, then he chose $S[\pointed a] \in \nd \pointed a$, and

\item\label{ItDoubleTwins} if there there are two pairs of twins $( \pointed a , \pointed b )$ and $( \pointed a' , \pointed b' )$ both of height $r$
 in $\lft(\eta)\circ \rgt(\eta)$ and $r$ distinguishes $\pointed a$ and $\pointed a'$, 
then Hercules did not play a $\pd$-move at $\eta$.

\end{enumerate}

\end{lemma}

\begin{proof}
Assume that Hercules played either a $\pd$-move or a $\nd$-move, and let $\eta'$ be the new head that was created.\\

\noindent {\sc Claim} \ref{ItTwinsDiamond}. If $\lft(\eta)\circ\rgt(\eta)$ contains twins $(\pointed a,\pointed b)$ and Hercules plays a $\pd$-move in $\eta$, he must choose $\pointed a'\in \nd \pointed a$ to place in $\lft (\eta')$. If $\pointed a'\not=S[\pointed a]$, then by Lemma \ref{LemmBoxmove}\ref{ItSpecialPairs}, there is $\pointed b'\in \nd \pointed b$ such that $\pointed a'\bis \pointed b'$. Since the Hydra plays greedily, we have that $\pointed b' \in \rgt(\eta')$, which by Lemma \ref{LemBisimLose} implies that Hercules cannot win.\\

\noindent {\sc Claim} \ref{ItTwinsBox}. This is simliar to the previous item. If Hercules plays a $\nd$-move in $\eta$, then he must choose $\pointed b' \in \nd \pointed b$ to place in $\lft (\eta')$. But then, by Lemma \ref{LemmBoxmove}\ref{ItBoxOne}, the Hydra will place a bisimilar $\pointed a'\in \rgt(\eta')$, and Hercules cannot win.\\

\noindent {\sc Claim} \ref{ItDoubleTwins}. Assume that there there are two pairs of twins $(S^r[\pointed a^n_i],S^r[\pointed b^n_i])$ and $(S^r[\pointed a^n_j],S^r[\pointed b^n_j])$ with $i<j$ in $\lft(\eta)\circ \rgt(\eta)$, such that $r$ distinguishes $\frm A^n_i$ and $\frm A^n_j$.  If Hercules plays a $\nd$-move, by claim \ref{ItTwinsDiamond}, he must place $S^{r+1}[\pointed a^n_i] \in \lft (\eta')$. By Lemma \ref{LemmBoxmove}\ref{ItBoxTwo}, there will be $\pointed v \in \nd S^r[\pointed b^n_j]$ such that $S^{r+1}[\pointed a^n_i] \bis \pointed v$. As before, this causes there to be bisimilar pointed models in $\lft(\eta')$ and $\rgt(\eta')$, which implies that Hercules cannot win. 
\end{proof}

Since the respective rightmost branches in the pointed models  $\pointed a^n_j$ 
and $\pointed b^n_j$ differ on a literal 
only in their leaves, we see that for every pair of twins $(\pointed a,\pointed b)$, the Hydra's strategy forces Hercules  to make 
$m$ many $\pd$-moves, where $m$ is the critical height  
of $\pointed a$ and $\pointed b$. Let us make this precise.

\begin{definition}\label{DefLambda}
Fix $n\geq 1$, $i\in [1,2^n]$ and a closed game tree $(T,\peq)$. 
Then, define $\Lambda(i)$ to be the set of leaves $\eta$ of $T$ such that for every $\eta'\peq \eta$, there is some $r \geq 0$ such that $(S^r[\pointed a^n_i],S^r[\pointed b^n_i])$ appear in $\lft(\eta') \circ \rgt(\eta')$.
\end{definition}

The sets $\Lambda(i)$ are non-empty and disjoint when Hydra plays amazingly, from which our exponential lower bound will follow immediately. To prove this, we will need the following lemma.

\begin{lemma}\label{LemmDiamNum}
Fix $n\geq 1$. Let $T$ be a closed game-tree for the $(\varphi_n,\cls{GL})$-{\meg} where the Hydra plays amazingly. Let $i\in[1,2^n]$, and $\eta \in \Gamma(i)$. Then,

\newcounter{counter}

\begin{enumerate}[label={\rm (\alph*)}]

\item\label{ItExR} for all $\zeta\peq \eta$ and all $r\geq 0$, if Hercules has played $r$ $\pd$-moves before $\zeta$ then $(S^r[\pointed a^n_i],S^r[\pointed b^n_i])$ appear in $\lft(\zeta) \circ \rgt (\zeta)$, and

\item\label{ItExM} if Hercules played $m$ $\pd$-moves before $\eta$, then $m$ is the critical height of $\frm A^n_i$ and $\frm B^n_i$.

\setcounter{counter}{\value{enumi}}

\end{enumerate}

\end{lemma}

\begin{proof}
Assume that the Hydra played amazingly, and that a closed game tree $T$ with root $\eta_0$ is given.\\

\noindent{\sc Claim} \ref{ItExR}. We proceed by induction on $\zeta$ along $\prec$. For the induction to go through, we need to prove a slightly stronger claim: if Hercules has played $r$ $\pd$-moves before $\zeta$, then $(S^r[\pointed a^n_i],S^r[\pointed b^n_i])$ appear in $\lft(\zeta) \circ \rgt (\zeta)$, and
\begin{enumerate}[label=(\alph*)]
\setcounter{enumi}{\value{counter}}

\item\label{ItUnique} for all $t\not = r$, $S^{t}[\pointed a^n_i] \not\in \lft(\zeta)$.
\end{enumerate}
For the base case this is clear, as only $S^0[\pointed a^n_i],S^0[\pointed b^n_i]$ appear in $\cls A^n \circ \cls B^n  =\lft(\eta_0) \circ \rgt(\eta_0)$, and Hercules has played zero $\pd$-moves before $\eta_0$.

For the inductive step, assume the claim for $\zeta$, and suppose that $\zeta'\peq \eta$ is a daughter of $\zeta$; we will prove claims \ref{ItExR} and \ref{ItUnique} for $\zeta'$. Let $r$ be the number of $\pd$-moves that Hercules has played before $\zeta$.  Since $\eta\in \Lambda(i)$, we have that $(S^k[\pointed a^n_i],S^k[\pointed b^n_i])$ occur in $\lft(\zeta')\circ\rgt(\zeta')$ for some $k$.

Obviously Hercules did not play a literal move on $\zeta$, or it would be a leaf. If Hercules played a $\vee$- or $\wedge$-move, since these moves do not introduce new pointed models, it follows that $(S^k[\pointed a^n_i],S^k[\pointed b^n_i])$ also occur in $\lft(\zeta)\circ\rgt(\zeta)$, and by uniqueness that $k=r$, from which claim \ref{ItExR} follows for $\zeta'$. As for claim \ref{ItUnique}, if $S^t[\pointed a^n_i] \in \lft(\zeta')$, then once again we have that $S^t[\pointed a^n_i] \in \lft(\zeta)$ and thus $t = r$.

If Hercules played a $\pd$-move, then $(S^r[\pointed a^n_i],S^r[\pointed b^n_i])$ occur in $\lft(\zeta)\circ\rgt(\zeta)$ by the induction hypothesis. By Lemma \ref{lemm:moves}\ref{ItSpecialPairs}, Hercules chose $S^{r+1}[\pointed a^n_i] \in \nd S^r[\pointed a^n_i]$. By Lemma \ref{LemmDistinguishAB}, $\frm A^n_i$ has the same critical height as $\frm B^n_i$, and thus $S^{r+1}[\pointed b^n_i] $ is defined. Since Hydra plays greedily, she chose $S^{r+1}[\pointed b^n_i] \in \nd S^r[\pointed b^n_i].$ But, there are now $r + 1$ $\pd$-moves before $\zeta'$, so claim \ref{ItExR} follows. Moreover, if $S^t[\pointed a^n_i] \in \lft(\zeta')$, then there must be $\pointed a' \in \lft(\zeta)$ such that $S^t[\pointed a^n_i] \in \nd\pointed a'$. But, since $\frm A^n_i$ is a tree, this can only occur when $\pointed a' = S^{t'}[\pointed a^n_i]$ for some $t' < t$, and it follows that $t'=r$ by the induction hypothesis, so that once again by Lemma \ref{lemm:moves}\ref{ItSpecialPairs}, $t = r+1$. Claim \ref{ItUnique} follows.

Finally, we note that Hercules cannot play a $\nd$-move on $\zeta$ by Lemma \ref{lemm:moves}\ref{ItTwinsBox}.\\

\noindent {\sc Claim} \ref{ItExM}. 
Let $r$ be the number of $\pd$-moves that Hercules played before $\eta$. By Lemma \ref{LemmDistinguishAB}, if $m$ is the critical height of $\frm A^n_i$, then it is also the critical height of $\frm B^n_i$ and $m$ distinguishes $\frm A^n_i$ and $\frm B^n_i$. Since $T$ is closed, $\eta$ must be a stub, which means that Hercules must have played a literal move on $\eta$. But this is only possible if $S^r[\pointed a^n_i]$ and $S^r[\pointed b^n_i]$ disagree on a literal, which is only possible if $r = m$.
\end{proof}

\begin{lemma}\label{lemmSurjection}
Fix $n\geq 1$. Let $(T,\peq)$ be a closed game-tree for the $(\varphi_n,\cls{GL})$-{\meg} where the Hydra plays amazingly.  Then,

\begin{enumerate}[label=(\alph*)]

\item\label{ItSurjNonEmpty} For all $i\in[1,2^n]$, $\Lambda(i)$ is non-empty,

\item\label{ItInjective} if $1\leq i< j \leq 2^n$, then $\Lambda(i) \cap \Lambda(j) = \varnothing$.

\end{enumerate}

\end{lemma}

\begin{proof} Let $\pointed a=\pointed a^n_i$ and $\pointed b=\pointed b^n_i$.\\

\noindent\ref{ItSurjNonEmpty}  We show by induction on the number of rounds in the game that there is always a leaf $\eta$ such that
\begin{itemize}

\item[(\textasteriskcentered)] for all $\zeta\peq \eta$ there is $r\geq 0$ such that $(S^r[\pointed a],S^r[\pointed b])$ appear in $\lft (\zeta) \circ \rgt (\zeta)$.

\end{itemize} 
 For the base case, we take $\eta$ to be the root, in which case it is clear that $(S^0[\pointed a],S^0[\pointed b])$ appears in $\lft (\eta) \circ \rgt (\eta) = \cls A^n\circ \cls B^n$. For the inductive step, assume that $\eta$ is a leaf such that (\textasteriskcentered) holds. We may assume that Hercules plays on $\eta$, for otherwise $\eta$ remains on the game-tree as a leaf.

If Hercules plays a literal move, then $\eta$ simply becomes a stub, but remains on the game-tree. If Hercules plays a $\vee$- or $\wedge$-move, then two heads $\eta_1$ and $\eta_2$ are added, and as in the proof of Lemma \ref{LemBisimLose}, either $(S^r[\pointed a],S^r[\pointed b])$ occurs in $\lft (\eta_1) \circ \rgt(\eta_1)$ and we take $\eta_1$ as the new head, or it occurs in $\lft(\eta_2) \circ \rgt(\eta_2) $ and we take $\eta_2$ instead.

If Hercules plays a $\pd$-move, then a new node $\eta'$ is added, and by Lemma \ref{lemm:moves}\ref{ItSpecialPairs}, Hercules places $S^{r+1}[\pointed a]$ in $\lft (\eta')$. Since the Hydra plays greedily and $S^{r+1}[\pointed b]$ exists by Lemma \ref{LemmDistinguishAB}, we have that $S^{r+1}[\pointed b] \in \rgt (\eta')$. Therefore, the twins $(S^{r+1}[\pointed a],S^{r+1}[\pointed b])$ appear in $\lft(\eta')\circ \rgt(\eta')$. Finally, Hercules cannot play a $\nd$-move by Lemma \ref{lemm:moves}\ref{ItTwinsBox}.\\

\noindent\ref{ItInjective} Now, let $1\leq i < j\leq 2^n$. Towards a contradiction, assume that $\eta\in \Lambda(i) \cap \Lambda(j)$. Let $m$ be the number of $\pd$-moves that Hercules played before $\eta$. By Lemma \ref{LemmDiamNum}\ref{ItExM}, $\frm A^n_i$ and $\frm A^n_j$ both have critical height $m$. By Lemma \ref{LemmDistinguishAA}, there is $r<m$ that distinguishes $\frm A^n_i$ and $\frm A^n_j$. Let $\zeta'\peq \eta$ be the first node such that Hercules has played $r+1$ $\pd$-moves before $\zeta'$, and $\zeta$ be its predecessor. Then, by Lemma \ref{LemmDiamNum}\ref{ItExR}, $(S^r[\pointed a^n_i],S^r[\pointed b^n_i])$ and $(S^r[\pointed a^n_j],S^r[\pointed b^n_j])$ both appear on $\lft(\zeta) \circ \rgt (\zeta)$, which by Lemma \ref{lemm:moves}\ref{ItDoubleTwins} implies that Hercules cannot play a $\pd$-move at $\zeta$. This means that he cannot have played $r+1$ $\pd$-moves before $\zeta'$, a contradiction.
\end{proof}

We are finally ready to prove our lower bound on the number of moves in the $(\varphi_n,\cls{GL})$-{\meg}. In fact, we have proven a slightly stronger claim.

\begin{proposition}\label{propSuccGL} 
For every $n\geq 1$, Hercules has no winning strategy of less than $2^n$ moves in the $(\varphi_n,\cls A^n\cup \cls B^n)$-\meg.
\end{proposition}

\proof
Assume that Hydra plays amazingly, and let $T$ be a closed game tree. Then, by Lemma \ref{lemmSurjection}, the sets of leaves $\{\Lambda(i) : i \in [1,2^n]\}$ are non-empty and disjoint. It follows that there are at least $2^n$ leaves, and since closing each leaf requires one literal move, Hercules must have played at least $2^n$ moves.
\endproof

Since $\cls A^n\cup \cls B^n\subseteq \cls{GL}$, Theorem \ref{thm:succinctness}\ref{ItTheoSuccA} readily follows. In view of Theorem \ref{thrm: satisfactionGames}, we also obtain the following stronger form of Proposition \ref{PropSuccinctnessBound}\ref{ItSuccinctnessGL}:

\begin{proposition}\label{propBoundBasicLangGL}
For all $n\geq 1$, whenever $\psi \in \lang_{\pd}$ is such that $\varphi_n \equiv \psi$ on $\cls A^n \cup  \cls B^n  $, it follows that $|\psi| \geq 2^n$.
\end{proposition}

\subsection{ $\lang_\ps$ is exponentially more succinct than $\lang_\pd$ on $\cls{TC}$}\label{subsectPdonRn}

We proceed to show that Hercules has no winning strategy of less than $2^n$ moves in $(\varphi_n, \cls{TC})$-{\meg}.
We begin by defining two sets of pointed models $\hat{\cls{A}}^n$ and $\hat{\cls{B}}^n$
that are a slight modification of the models in $\cls{A}^n$ and $\cls{B}^n$, respectively.

\begin{definition}
Let $\frm K=(\dom{\frm K},R_\frm K,V_\frm K)$ be any Kripke model. We define a new model $\hat {\frm K}$ such that
\begin{enumerate}[label=(\roman*)]

\item  $\dom{\hat{\frm K}}=\dom{\frm K}\cup \{\infty\}$, where $\infty\not\in \dom{\frm K}$,

\item $R_{\hat{\frm K}}=R_\frm K\cup (\dom{\hat{\frm K}}\times \{\infty\} )$, and

\item $V_{\hat{\frm K}}(w)=V_{{\frm K}}(w)$ if $w\in \dom{\frm K}$, $V_{\hat{\frm K}}(\infty)=\varnothing$.

\end{enumerate}
If $\frm K$ is equipped with a successor partial function $S_\frm K$, we also define $S_{\hat{\frm K}}=S_\frm K $.
For a class of models $\cls X$, we denote $\{\hat {\frm K} : \frm K\in \cls X\}$ by $\hat{\cls X}$.
\end{definition}

In other words, we add a `point at infinity' that is seen by both worlds.
Note that the successor function remains unchanged, i.e.~$\infty$ is never a successor. 
In particular, $\infty$ can never belong to a critical branch.

This operation allows us to easily turn a model into a totally connected model:

\begin{lemma}\label{LemmInfty}
Let $\frm M,\frm N$ be $\cls{K4}$ models. Then:
\begin{enumerate}[label={\rm (\alph*)}]

\item $\hat{\frm M}$ is a $\cls{TC}$ model;

\item if $w\in \dom{\frm M}$ and $v\in \dom{\frm N}$ are such that $(\frm M,w)\bis(\frm N,v)$, then $(\hat{\frm M},w)\bis(\hat{\frm N},v)$, and

\item $(\hat{\frm M},\infty) \bis (\hat{\frm N},\infty)$.

\end{enumerate}
\end{lemma}
\proof
If $w,w'\in \dom{\hat{\frm M}}$, then $w,\infty,w'$ is a path connecting $w$ to $w'$ (as $x\mathrel R_{\hat{\frm M}} \infty$ holds for all $x\in\dom{\hat{\frm M}}$). It follows that $\hat{\frm M}$ is connected, and indeed the same path witnesses that $\hat{\frm M}$ is locally connected if we take $w,w'\in R_\frm M(u)$. The second claim follows from observing that if $\chi \subseteq \dom{\frm M} \times \dom{\frm N}$ is a bisimulation then so is $\hat\chi = \chi \cup \{(\infty,\infty)\} \subseteq \dom{\hat{\frm M}} \times \dom{\hat{\frm N}}$, and the third by choosing an arbitrary such $\hat \chi$ (say, $\hat \varnothing$).
\endproof

\begin{figure}
\begin{center}

\begin{tikzpicture}
[
help/.style={circle,draw=white!100,fill=white!100,thick, inner sep=0pt,minimum size=1mm},
white/.style={circle,draw=black!100,fill=white!100,thick, inner sep=0pt,minimum size=2mm},
]

 \node at (1, 16) [white](1A1) [label=270: ${}_1$, label=180: $\triangleright$]{};
\node at (1, 15.25)[help](1A11)[label=270:$\bm a_1^1 $]{};
 \node at (2, 16) [white](1A2) [label=180:$\triangleright$]{};
\node at (2.25, 15.25)[help](1A2name)[label=270:$\bm a_2^2 $]{};
 \node at (2.75, 17) [white](2A3) [label=0: ${}_1$]{};
\node at (2, 17) [white](2A2) []{};

\node at (4.5, 16) [white](1B1) [label=180:$\triangleright$ ]{};
\node at (4.5, 15.25)[help](1B11)[label=270:$\bm b_1^2 $]{};
 \node at (5.5, 16) [white](1B2) [label=180:$\triangleright$]{};
\node at (5.75, 15.25)[help](1B2name)[label=270:$\bm b_2^2 $]{};
 \node at (5.5, 17) [white](2B2) []{};

\node at (1, 17) [white](reflexiveA1) []{};
\node at (2.75, 18) [white](reflexiveA2) []{};
\node at (4.5, 17) [white](reflexiveB1) []{};
\node at (5.5, 18) [white](reflexiveB2) []{};

\node at (3.5, 19)[help](helpdottedup)[]{};
\node at (3.5, 15)[help](helpdotteddown)[]{};

\begin{scope}[>=stealth, auto]

\draw [->] (reflexiveA1) to [in=50, out=130, loop]  (reflexiveA1);
\draw [->] (reflexiveA2) to [in=50, out=130, loop]  (reflexiveA2);
\draw [->] (reflexiveB1) to [in=50, out=130, loop]  (reflexiveB1);
\draw [->] (reflexiveB2) to [in=50, out=130, loop]  (reflexiveB2);

\draw [->, bend left] (1B2) to   (reflexiveB2);

\draw [->] (1A1) to  (reflexiveA1);
\draw [->] (1A2) to  (2A2);
\draw [->] (1A2) to  (2A3);
\draw [->] (2A2) to  (reflexiveA2);
\draw [->] (2A3) to  (reflexiveA2);
\draw [->] (1A2) to  (reflexiveA2);

\draw [->] (1B1) to  (reflexiveB1);
\draw [->] (1B2) to  (2B2);
\draw [->] (2B2) to  (reflexiveB2);

\draw [-, dotted] (helpdottedup) to  (helpdotteddown);

\end{scope}

\end{tikzpicture}

\end{center}
\caption{The pointed models in $\hat{\cls{A}}^1$ and $\hat{\cls{B}}^1$
are shown on the left and on the right of the dotted line, respectively.}
\label{fig:succinctnessModelsRn1}
\end{figure}
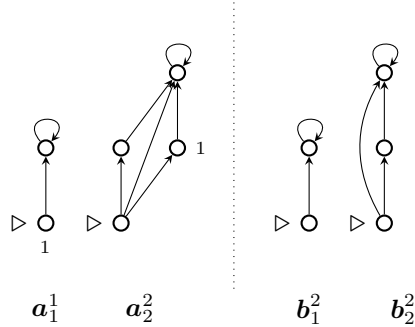

We have the following analogue of Lemma \ref{LemmPhiTF}. The proof is identical and we omit it.

\begin{lemma}\label{LemmInfMove}
For any $n\geq 1$, $\hat{\cls{A}}^n\models \varphi_n$ and $\hat{\cls{B}}^n \models \neg \varphi_n$.
\end{lemma}

Intuitively,  the proof that  Hercules has no winning strategy of less than $2^n$ moves in the game that 
starts with $\hat{\cls{A}}^n\circ\hat{\cls{B}}^n$  revolves around  the observation that 
 the models in $\hat{\cls{A}}^n$ and $\hat{\cls{B}}^n$ are 
constructed in such a way that Hercules, when playing 
a $\pd$- or $\nd$-move at some position labeled by $\cls{L}\circ\cls{R}$,  cannot pick $\infty$ in any model of  
$\cls L$ or $\cls R$, as this will lead to bisimilar pointed models on each side. Thus, given the construction of the models in 
 $\hat{\cls{A}}^n$ and $\hat{\cls{B}}^n$, Hercules and the Hydra are essentially playing a  
$(\varphi_n, \cls{GL}) $-{\meg} on  $\cls{A}^n\circ\cls{B}^n$ and the lower bound
on the number of moves in any winning $(\varphi_n, \cls{GL}) $-{\meg} for Hercules on    $\cls{A}^n\circ\cls{B}^n$
established in the previous sub-section applies to the present case too.

Let us formalise the above intuitive considerations. As before, we will say that the Hydra plays amazingly in the $(\varphi_n,\cls{TC})$-{\meg} if she labels the root by $\hat{\cls A}^n \circ \hat{\cls B}^n$ and plays greedily.

 \begin{lemma}\label{lemm:infty}
Fix $n\geq 1$, and assume that the Hydra plays amazingly in the $(\varphi_n,\cls{TC})$-{\meg}.

Then, for any position $\eta$ in a closed game tree $T$ such that $\lft(\eta)$ and $\rgt(\eta)$ are both non-empty,
\begin{enumerate}[label=(\roman*)]

\item if Hercules  played a $\pd$-move in $\eta$, he did not pick any pointed model of the form $(\frm A,\infty)$, and

\item if Hercules  played a $\nd$-move in $\eta$, he did not pick any pointed model of the form $(\frm B,\infty)$.

\end{enumerate}
\end{lemma}
\begin{proof}
The claims are symmetric, so we prove the first. If Hercules picked a frame of the form $(\frm A,\infty)$, since the Hydra plays greedily, she will also pick at least one frame of the form $(\frm B,\infty)$. By Lemma \ref{LemmInfty}, $(\frm A,\infty) \bis (\frm B,\infty)$, which in view of Lemma \ref{LemBisimLose} contradicts the assumption that $T$ is closed.
\end{proof}

We will not give a full proof of Theorem \ref{thm:succinctness}.\ref{ItTheoSuccB}, as it proceeds by replacing $\cls A^n\circ \cls B^n$ by $\hat{\cls A}^n \circ \hat{\cls B}^n$ throughout Section \ref{sectScattered}. Instead, we give a rough outline below.

\begin{proposition}\label{propSuccTC} 
For every $n\geq 1$, Hercules has no winning strategy of less than $2^n$ moves in the $(\varphi_n,\hat{\cls A}^n\cup \hat{\cls B}^n)$-\meg.
\end{proposition}

\proof[Proof sketch]
The analogues of Lemmas \ref{LemmDistinguishAA} and Lemmas \ref{LemmDistinguishAB} for $\hat{\cls A}^n \circ \hat{\cls B}^n$ follow directly from the original statements by observing that $S_{\hat{\frm K}} = S_{{\frm K}}$, so that the critical branch is identical. Similarly, an analogue of Lemma \ref{LemmBoxmove} follows easily from the original if we use Lemma \ref{LemmInfty} to see that, whenever $\pointed a'\bis \pointed b'$, we also have that $\hat{\pointed a}' \bis \hat{\pointed b}'$. The analogue of Lemma \ref{lemm:moves} can then be proven as before, using Lemma \ref{lemm:infty} to rule out situations where Hercules chooses $\infty$.

The rest of the results leading up to Proposition \ref{PropSuccinctnessBound} rely on these basic lemmas and thus readily apply to the $(\varphi_n,\cls{TC})$-{\meg}.
In particular, Definition \ref{DefLambda} makes no assumption about the frames appearing in $T$, hence the sets $\Lambda(i)$ for $i \in [1,2^n]$ are readily available for the $(\varphi_n,\cls{TC})$-{\meg}, and as before are disjoint and non-empty. We conclude that Hercules has no winning strategy in less than $2^n$ moves.
\endproof

Theorem \ref{thm:succinctness}\ref{ItTheoSuccB} readily follows, as does the following stronger version of Proposition \ref{PropSuccinctnessBound}\ref{ItSuccinctnessTC}:

\begin{proposition}\label{propBoundBasicLangTC}
For all $n\geq 1$, whenever $\psi \in \lang_{\pd}$ is such that $\varphi_n \equiv \psi$ on $\hat{\cls A}^n \cup \hat{ \cls B}^n  $, it follows that $|\psi| \geq 2^n$.
\end{proposition}

\section{Succinctness in the extended language}\label{SecSuccExt}

Proposition \ref{PropSuccinctnessBound} holds even if we replace $\lang_\pd$ by the extended language $\lang^{\pt}_{\pd\forall}$. In this section, we will first extend these results to $\lang _{\pd \forall}$ and then to $\lang^{\pt}_{\pd\forall}$.

\subsection{Succinctness with the universal modality}\label{subsectForall}

As it turns out, the universal modality is not of much help in expressing $\ps$ succinctly. This is perhaps not surprising, as $\ps$ is essentially a local operator and $\forall$ is global. The disjoint union operation from Section \ref{SecAmalg} will be useful in making this precise.

\begin{definition}\label{defModelsForallScatt}
Let the model $\frm{C}^n$ be obtained by taking the 
disjoint union of all the models used in    $\cls{A}^n$ and $\cls{B}^n$ i.e., 
\[
\frm{C}^n=\Big ( \coprod_{i=1}^{2^n} \frm A^n_i \Big ) \amalg  \Big ( \coprod_{i=1}^{2^n} \frm B^n_i \Big ).
\]
\end{definition}

The models $\frm{C}^n$ will allow us to make the universal modality effectively useless in distinguishing between the pointed models we have constructed.

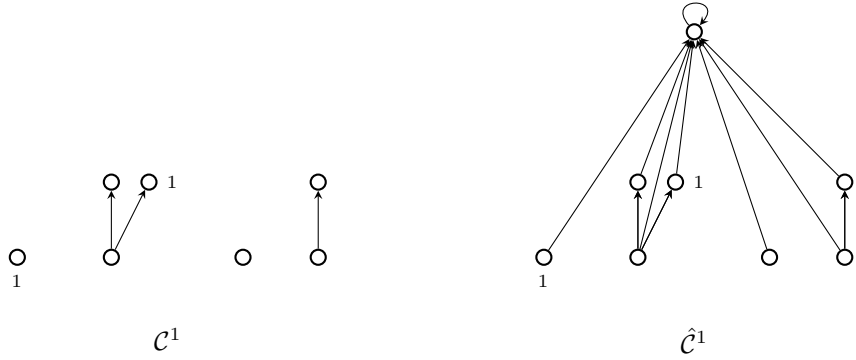
\begin{figure} 
\begin{center}

\begin{tikzpicture}
[
help/.style={circle,draw=white!100,fill=white!100,thick, inner sep=0pt,minimum size=1mm},
white/.style={circle,draw=black!100,fill=white!100,thick, inner sep=0pt,minimum size=2mm},
]

 \node at (1, 16) [white](1A1) [label=270: ${}_1$]{};

 \node at (2.25, 16) [white](1A2) []{};

 \node at (2.75, 17) [white](2A3) [label=0: ${}_1$]{};
\node at (2.25, 17) [white](2A2) []{};

\node at (4, 16) [white](1B1) [ ]{};

 \node at (5, 16) [white](1B2) []{};
 \node at (5, 17) [white](2B2) []{};

\begin{scope}[>=stealth, auto]

\draw [->] (1A2) to  (2A2);
\draw [->] (1A2) to  (2A3);

\draw [->] (1B2) to  (2B2);

\node at (3, 15.25)[help](1A11)[label=270:$\mathcal C^1$]{};


\node at (8, 16) [white](1A1r) [label=270: ${}_1$]{};

 \node at (9.25, 16) [white](1A2r) [ ]{};

 \node at (9.75, 17) [white](2A3r) [label=0: ${}_1$]{};
\node at (9.25, 17) [white](2A2r) []{};
\node at (10, 19) [white](reflexive1r) []{};
\node at (11, 16) [white](1B1r) [ ]{};
\node at (10, 15.25)[help](1B11r)[label=270:$\hat{\mathcal{C}}^1$]{};
 \node at (12, 16) [white](1B2r) []{};

 \node at (12, 17) [white](2B2r) []{};







\draw [->] (reflexive1r) to [in=50, out=130, loop]  (reflexive1r);

\draw [->] (1A1r) to  (reflexive1r);
\draw [->] (1A2r) to  (2A2r);
\draw [->] (1A2r) to  (2A3r);
\draw [->] (2A2r) to  (reflexive1r);
\draw [->] (2A3r) to  (reflexive1r);
\draw [->] (1A2r) to  (reflexive1r);

\draw [->] (1A2r) to  (2A2r);
\draw [->] (1A2r) to  (2A3r);

\draw [->] (1B1r) to  (reflexive1r);
\draw [->] (1B2r) to  (reflexive1r);
\draw [->] (1B2r) to  (2B2r);

\draw [->] (1B2r) to  (2B2r);

\draw [->] (2B2r) to  (reflexive1r);





\end{scope}

\end{tikzpicture}

\end{center}
\caption{The models $\frm C^1$ and $\hat{\frm C}^1$.}
\label{fig:succinctnessModels1forallScatt}
\end{figure}

\begin{proposition}\label{PropSuccinctnessBoundForall}
For all $n\geq 1$, whenever $\psi \in \lang_{\pd \forall}$ is such that $\varphi_n \equiv \psi$ either on $\mathcal C^n$ or on $\hat{\mathcal C}^n$, it follows that $|\psi| \geq 2^n$.
\end{proposition}

\begin{proof}
Given a model ${\frm M}$, consider a translation $\trC{\frm M} \colon \lang_{\pd\forall} \to \lang_{\pd}$ that
commutes with all Booleans and $\pd,\nd$, and so that if $\theta$ is of one of the forms  $\forall \varphi$ or $\exists\varphi$, then $\trC{{\frm M}}(\theta) = \top$ if ${\frm M}\models\theta$, $\trC{{\frm M}}(\theta) = \bot$ otherwise.
It is immediately clear that for $w\in \dom{{\frm M}}$ and any $\psi\in \lang_{\pd\forall}$, $({\frm M}, w)\models \psi$ if and only if $({\frm M}, w)\models \trC{\frm M}(\psi)$; moreover, it is obvious that $|\trC{\frm M}(\psi)|\leq |\psi|$. 

Let us consider a point $w$ in one the models $\frm A^n_j$ from Definition~\ref{def:models}.
Since $(\frm A^n_j,w)$ is locally bisimilar to $(\frm{C}^n,w)$, we have by Lemma \ref{LemmTruthBisim} that for any  $\lang_{\pd}$- formula $\psi$, $(\frm{C}^n,w)\models \psi$ if and only if $(\frm{A}^n_j,w)\models \psi$.

Thus, for any $\lang_{\pd\forall}$-formula $\psi$ that is equivalent to  $\varphi_n$  on $\mathcal C^n$, we have that its translation $\trC{\frm C^n}(\psi)$ is also equivalent to $\varphi_n$ on $\cls{A}^n\cup \cls{B}^n$.
Note, however, that according to Proposition~\ref{propBoundBasicLangGL}, the size of $\trC{\frm C^n}(\psi)$ is at least $2^n$. This establishes the proposition for $\mathcal C^n$.

For $\hat{\frm C}^n$ we proceed as above, but work instead with Proposition \ref{propBoundBasicLangTC}.
\end{proof}

\subsection{Succinctness with tangle and fixed points}

Recall that the translation $\trtanpd(\varphi)$ is defined simply by replacing every occurrence of $\pt\Phi$ by $\bot$ and every occurrence of $\nt\Phi$ by $\top$.
Since Proposition \ref{PropSuccinctnessBound} applies to scattered spaces, we can use our work in Section \ref{SecScattered} to immediately obtain succinctness results relative to the tangled derivative over such spaces.

For metric spaces, however, the behavior of the tangled derivative is less trivial. Fortunately, in the $\cls{TC}$ models we have constructed, its behavior is still rather simple.

\begin{lemma}\label{LemmTangleHat}
Let $\frm K$ be a $\cls{GL}$ model and $\Phi \subseteq \lang^{\pt}_{\pd\forall}$ be finite. Then, for any $w\in \dom{\hat{\frm K}}$,
\begin{enumerate}[label=(\alph*)]

\item $(\hat{\frm K},w) \models \pt\Phi$ if and only if $(\hat{\frm K},\infty) \models \bigwedge \Phi$, and

\item $(\hat{\frm K},w) \models \nt\Phi$ if and only if $(\hat{\frm K},\infty) \models \bigvee \Phi$.

\end{enumerate}

\end{lemma}

\begin{proof}
Let $\frm K$ be any $\cls{GL}$ model. By the semantics of $\pt$, if $w\in\dom{\hat{\frm K}}$ and $\Phi \subseteq \lang^{\pt}_{\pd\forall}$ is finite, then $(\hat{\frm K},w)\models\pt\Phi$ if and only if there is an infinite sequence \[w=w_0 \mathrel R_{\hat{\frm K}} w_1 \mathrel R_{\hat{\frm K}} w_2 \mathrel \hdots\]
such that each formula of $\Phi$ holds on $(\hat{\frm K},w_n)$ for infinitely many values of $w_n$. However, since $\frm K$ is a $\cls{GL}$ frame, we must have $w_n = \infty$ for some value of $n$, which implies that $w_m = \infty$ for all $m\geq n$. From this it readily follows that $(\hat{\frm K},\infty)\models \bigwedge \Phi$. Conversely, if $(\hat{\frm K},\infty)\models \bigwedge \Phi$, then the sequence $\infty,\infty,\infty,\hdots$ witnesses that $(\hat{\frm K},w )\models \pt\Phi$. It follows that $(\frm M,w)\models\pt\Phi$ if and only if $(\frm M,\infty)\models\bigwedge\Phi$. By similar reasoning, $(\frm M,w)\models\nt\Phi$ if and only if $(\frm M,\infty )\models \bigvee\Phi$.
\end{proof}

This allows us to define a simple translation from $\lang^{\pt}_{\pd\forall}$ to $\lang_{\pd\forall}$ tailored for our $\cls{TC}$ models.

\begin{definition}
Fix a $\cls{GL}$ model $\frm K$.
We define a translation $\trtanwedge{\frm K} \colon \lang^{\pt}_{\pd\forall}\to \lang_{\pd\forall}$ by letting $\trtanwedge{\frm K}$ commute with Booleans and all modalities except $\pt,\nt$, in which case we set
\begin{center}
\phantom{a}
\hfill
$
\trtanwedge{\frm K}(\pt \Phi) =
\begin{cases}
\top&\text{if $(\hat{\frm K},\infty) \models \bigwedge\Phi$,}\\
\bot&\text{otherwise;}\\
\end{cases}
$
\hfill
$
\trtanwedge{\frm K}(\nt \Phi) =
\begin{cases}
\top&\text{if $(\hat{\frm K},\infty) \models \bigvee\Phi$,}\\
\bot&\text{otherwise.}\\
\end{cases}
$
\hfill\phantom{a}
\end{center}
\end{definition}

\begin{lemma}\label{LemmTanMetTrans}
Let $\frm K$ be a $\cls{GL}$ model and $\varphi \in \lang^{\pt}_{\pd\forall}$ be finite. Then, $ \trtanwedge{\frm K} (\varphi) \equiv \varphi$ over $\hat{\frm K}$.
\end{lemma}

\begin{proof}
Immediate from Lemma \ref{LemmTangleHat} using a routine induction on $\varphi$.
\end{proof}

We are now ready to prove the full version of our first main result:

\begin{theorem}\label{TheoMain}
Let $\cls C$ be a class of convergence spaces that contains either:
\begin{enumerate}

\item\label{ItMainGL} the class of all finite $\cls{GL}$ frames,

\item\label{ItMainTC} the class of all finite $\cls{TC}$ frames,

\item\label{ItMainomega} the set of ordinals $\Lambda<\omega^\omega$, or

\item\label{ItMainX} any crowded metric space $\frm X$.

\end{enumerate}
Then, there exist arbitrarily large $\varphi \in \lang_\ps$ such that, whenever $\psi \in \lang^{\pt}_{\pd\forall}$ is equivalent to $\varphi$ over $\cls C$, it follows that $|\psi| \geq 2^{\frac{|\varphi|}3}$.
\end{theorem}

\begin{proof}
First assume that $\cls C$ contains all finite $\cls{GL}$ frames. Fix $n\geq 1$, and assume that $\psi \in \lang^{\pt}_{\pd\forall}$ is equivalent to $\varphi_n$ over $\cls{C}$. Then, by Corollary \ref{CorTanPd}, $\trtanpd(\psi) \in \lang_{\pd\forall}$ is equivalent to $\psi$, and hence to $\varphi_n$, over $\cls C$, and in particular, over $\frm C^n$. By Proposition \ref{PropSuccinctnessBoundForall}, it follows that  $|\trtanpd(\psi)|\geq 2^n$, and hence $|\psi| \geq 2^n$ as well. This establishes item \ref{ItMainGL} and, in view of Corollary \ref{CorOrd}, item \ref{ItMainomega}.

If $\cls C$ contains the class of all finite $\cls{TC}$ frames, we proceed as above, but instead use the translation $\trtanwedge{{\frm C}^n}$, so that $\trtanwedge{{\frm C}^n}(\psi) \in \lang_{\pd\forall}$ is equivalent to $\psi \in \lang^{\pt}_{\pd\forall}$, and hence to $\varphi_n$, over $\hat{\frm C}^n$. Finally, we use Corollary \ref{CorRn} to lift this result to $\cls C$ containing any crowded metric space $\frm X$.
\end{proof}

Given the fact that $\lang^{\pt}_\pd$ is equally expressive as $\lang^\mu_\pd$, it is natural to ask which is more succinct. Note that $\lang^{\pt}_\pd$ cannot be exponentially more succinct than $\lang^\mu_\pd$, as the translation $\trtanb$ is polynomial. On the other hand, we do have that the $\mu$-calculus is more succinct than the tangled language:

\begin{theorem}\label{TheoMuSucc}
Let $\cls C$ contain either:
\begin{enumerate}

\item the class of all finite $\cls{GL}$ frames,

\item the class of all finite $\cls{TC}$ frames,

\item the set of ordinals $\Lambda<\omega^\omega$, or

\item any crowded metric space $\frm X$.

\end{enumerate}
Then, there exist arbitrarily large $\theta \in \lang^\mu_\pd$ such that, whenever $\psi \in \lang^{\pt}_{\pd\forall}$ is equivalent to $\theta$ over $\cls C$, it follows that $|\psi|> 2^{\frac{|\theta|}{12}}$.
\end{theorem}

\begin{proof}
By Lemma \ref{LemmMuShort}, for all $n \in \mathbb N$, $\varphi_n \equiv \trpsmu (\varphi_n)$ and $|\trpsmu (\varphi_n)| \leq 4 |\varphi_n|$. Hence, the sequence $(\trpsmu (\varphi_n))_{n\in \mathbb N}$ witnesses that $\lang^\mu_{\pd}$ is exponentially more succinct than $\lang^{\pt}_{\pd\forall}$.
\end{proof}

\section{Concluding remarks}\label{SecConc}

There are several criteria to take into account when choosing the `right' modal logic for spatial reasoning. It has long been known that the limit-point operator leads to a more expressive language than the closure operator does, making the former seem like a better choice of primitive symbol. However, our main results show that one incurs in losses with respect to formula-size, and since the blow-up is exponential, this could lead to situations where e.g.~formally proving a theorem expressed with the closure operator is feasible, but treating its limit-operator equivalent is not.
Similarly, the results of \cite{DawarOtto} make the tangled limit operator seem like an appealing alternative to the spatial $\mu$-calculus, given its simpler syntax and more transparent semantics. Unfortunately, the price to pay is also an exponential blow-up.

We believe that the takeaway is that different modal logics may be suitable for different applications, and hope that the work presented here can be instrumental in clarifying the advantages and disadvantages of each option.
Moreover, there are many related questions that remain open and could give us a more complete picture of the relation between such languages.

For example, one modality that also captures interesting spatial properties is the `difference' modality, where ${\langle {\neq}\rangle}\varphi$ holds at $w$ if there is $v\not = w$ satisfying $\varphi$. This modality has been studied in a topological setting by Kudinov \cite{KudinovDifference}, and succinctness between languages such as $\lang_{\pd \forall}$ and $\lang_{\ps \langle{\neq}\rangle}$ remains largely unexplored. Even closer to the present work is the {\em tangled closure} operator, $\ps^\ast$, defined analogously to $\pt$, but instead using the closure operation. The techniques we have developed here do not settle whether $\lang^{\mu}_\ps$ is exponentially more succinct than $\lang^{\ps^\ast}_\ps$.

There are also possible refinements of our results. Our construction uses infinitely many variables, and it is unclear if $\lang_\ps$ is still exponentially more succinct than $\lang_\pd$ when restricted to a finite number of variables.
Finally, note that Theorem \ref{TheoMain} is sharp in the sense that an exponential translation is available, but this is not so clear for Theorem \ref{TheoMuSucc}, in part because an explicit translation is not given in \cite{DawarOtto}. Sharp upper and lower bounds remain to be found.



\section*{Acknowledgements}

This work was partially supported by ANR-11-LABX-0040-CIMI within the program ANR-11-IDEX-0002-02.

\end{document}